\documentclass[oneside, 12pt]{amsart}
\usepackage{amscd}
\usepackage{amssymb}
\usepackage{amsfonts}
\usepackage{remark}
\usepackage[normalem]{ulem}
\usepackage[backref=page]{hyperref}

\usepackage{enumerate}

\usepackage{url}

\numberwithin{equation}{section}
\newtheorem{thm}{Theorem}[section]
\newtheorem{theorem}[thm]{Theorem}

\newtheorem{claim}[thm]{Claim}
\newtheorem{lemma}[thm]{Lemma}
\newtheorem{corollary}[thm]{Corollary}
\newtheorem{proposition}[thm]{Proposition}
\newremark{definition}[thm]{Definition}
\newremark{ques}[thm]{Question}
\newremark{question}[thm]{Question}
\newremark{conj}[thm]{Conjecture}
\newremark{remark}[thm]{Remark}
\newremark{conjecture}[thm]{Conjecture}
\newremark{convention}[thm]{Convention}
\newremark{example}[thm]{Example}
\newremark{emp}[thm]{\kern -4pt}

\newcommand\dell{t}

\newcommand{\Spec}{\operatorname{Spec}}
\newcommand{\Jac}{\operatorname{Jac}}
\newcommand{\Pic}{\mbox{\rm Pic}\kern 1pt}
\newcommand{\Br}{\mbox{\rm Br}\kern 1pt}

\newcommand{\cO}{{\mathcal O}}

\newcommand{\Res}{\mathrm{Res}}

\newcommand{\Coker}{\operatorname{Coker}\kern 1pt}

\newcommand{\A}{\mathbb A}

\DeclareFontEncoding{OT2}{}{} 

\numberwithin{equation}{section}
\numberwithin{figure}{section}
\numberwithin{table}{section}
\makeatletter
    \let\c@equation\c@thm
    \let\c@figure\c@thm
    \let\c@table\c@thm
\makeatother

\begin{document}

\date{\today}
\title{New Points on Curves} 

\author{Qing Liu and Dino Lorenzini}
\address{School of Mathematical Sciences,
Xiamen University, 361005 Xiamen, China} 
\address{Universit\'e de Bordeaux, Institut de Math\'ematiques de 
Bordeaux, 33405 Talence, France} 
\email{Qing.Liu@math.u-bordeaux.fr}
\address{Department of mathematics, University of Georgia, 
Athens, GA 30602, USA} 
\email{lorenzin@uga.edu}

\thanks{Liu is partially supported by 
 the National Nature Science Foundation of China (Grant No. 11571286). Lorenzini is partially supported by NSA Grant H98230-15-1-0028 and by a grant from the Simons Foundation (Grant 245522, DL)}

\begin{abstract}
 Let $K$  be a  field and let $L/K$ be a finite extension. 
Let $X/K$ be a scheme of finite type. A point of $X(L)$ is said to be {\it new} if it does not belong to $\cup_F X(F)$, where $F$ runs over all proper subfields $K \subseteq F \subset L$. Fix now an integer $g>0$ and a  finite separable extension $L/K$ of degree $d$. We investigate in this article whether there exists a smooth proper geometrically connected curve of genus $g$ with a new point in $X(L)$. We show for instance that if $K$ is infinite with ${\rm char}(K)\neq 2$ and $g \geq \lfloor d/4\rfloor$, then there exist infinitely many 
hyperelliptic curves $X/K$ of genus $g$, pairwise non-isomorphic over $\overline{K}$, and  with a new point in $X(L)$. When $1 \leq d \leq 10$, we show that there exist infinitely many elliptic curves $X/K$ with pairwise distinct $j$-invariants and with a new point in $X(L)$.

\vspace*{.3cm}
\noindent 
\begin{tiny}KEYWORDS \end{tiny} New point, smooth curve, elliptic curve, large degree. 

\noindent 
\begin{tiny}MATHEMATICS SUBJECT CLASSIFICATION: 11G05, 11G20, 11G30, 14G05, 14G25  \end{tiny}
\vspace*{.3cm}
\end{abstract} 
 
\maketitle

\begin{section}{Introduction } 

Let $K$  be a  field, and let $X/K$ be a scheme of finite type.
 Let $L/K$ be a finite extension. 
A point of $X(L)$ is said to be {\it new} if it does not belong to $\cup_F X(F)$, where $F$ runs 
 over all proper subfields $K \subseteq F \subset L$. In other words, $X/K$ has a {\it new point in $X(L)$}, or a {\it new point over $L$}, if there exists in $X$ a closed point $P$ whose residue field $K(P)$ is isomorphic over $K$ to $L$. With this definition, any point in $X(K)$ is new.

Fix a positive integer $g \geq 1$ and consider the following question. Given a finite extension $L/K$, is it possible to find a smooth proper
geometrically connected curve $X/K$ of genus $g$ such that $X$ has a
new point over $L$? As we shall see, the answer to this question
depends in an essential way on the properties of the ground field
$K$. It is not hard to show that this question always has a  positive answer when $K$ is finite (\ref{thm.finite}), 
and when $K$ is large (\ref{thm.PAC}). The latter hypothesis holds for instance when $K$ is the field of fractions of a Henselian discrete valuation ring, 
or when  $K$ is a pseudo-algebraically closed field.

We have not been able to answer the above question for all $g>0$ when $K$ is a number field, and it may be that in this case the question has a negative answer in general. In this article, we present some instances where we can prove the existence of infinitely many curves of genus $g>0$ with a new point over $L$. 
  Our result
below is valid for any infinite field $K$ of characteristic not equal to $2$.
The case where ${\rm char}(K)=2$ is treated in Section \ref{char2}.  Recall the notation 
$\lfloor x \rfloor$ for the floor of a rational number $x$.

\medskip
\noindent
{\bf Theorem (see \ref{pro.boundgeneral} and \ref{thm.bound}).} {\it Let $K$ be any
infinite field with ${\rm char}(K) \neq 2$. 
Let $L_1, \dots, L_t$ be finite separable extensions of $K$  
(the extensions $L_i$ need not be distinct). Let  
$d:=\sum_{i=1}^t [L_i:K] $ and assume that $d \geq 7$. 
Then for each $g \geq \lfloor d/4\rfloor$, there exist infinitely many 
hyperelliptic curves $X/K$ of genus $g$, pairwise non-isomorphic over
$\overline{K}$, and with distinct new points $P_1\in X(L_1), \dots, P_t\in
X(L_t)$ on $X$. }

\medskip
The statement of Theorem  \ref{pro.boundgeneral} is slightly more precise, and it is further refined in \ref{pro.bound} when the extensions $L_i/K$ 
contain a Kummer subextension. In the case of $g=1$, we obtain the following result.

\medskip
\noindent
{\bf Theorem (see \ref{cor.elliptic}).} {\it Let $K$ be an infinite field with ${\rm char}(K) \neq 2$. 
Let $L/K$ be a separable extension of degree $d $, with $ 1 \leq d \leq 10$. Then there exist infinitely many elliptic curves $X/K$  with distinct $j$-invariants and with a new point over $L$.

Fix an integer $N>1$. Then it is possible to find infinitely many such elliptic curves with a new point which
is either of order bigger than $N$  or of infinite order in $X(L)$.
When $K$ is a global field,
there exist infinitely many elliptic curves $X/K$ with distinct $j$-invariants and with a new point over $L$  of infinite order in $X(L)$. 
}
\medskip

When $K$ is a number field and   $d\leq 9 $, the existence of an elliptic curve with a new $L$-point of infinite order is proved by Rohrlich in \cite{Roh}, Theorem, by a different method.
A variation on Rohrlich's proof is given in \cite{Mat}, Corollary 5 (see also \ref{thm.d=9} in all characteristics).  
We discuss  in \ref{case ell=11} the case where $d =11$ and $L/{\mathbb Q}$ is a Kummer extension.
Abelian extensions of degree $12, 14, 15$, $20$, $21$, and $30$, 
are discussed in \ref{cor.extension15-16}.

Let $K$ be a number field and fix an integer $g \geq 2$. Theorem 1.2 of Caporaso, Harris, and Mazur, in \cite{CHM}, implies the existence of the  following integer $d(g)$ if the Strong Lang Conjecture holds (see \ref{CHM}):
$d(g)$ is the smallest integer
such that, 
given any finite extension $L/K$ with $[L:K]>d(g)$, 
then there exist only {\it finitely many} smooth proper geometrically connected curves $X/K$ of genus $g$ with 
a closed point $P$ having residue field isomorphic to $L$.
We show in \ref{cor.infinitetelymanycurves} that $d(g) > 6(g+1)$.
When $[L:K]>6(g+1)$, we do not know of any efficient general method to produce an example of a curve $X/K$ of genus $g\geq 2$
with a new point over $L$. 

The situation when $K$ is a finite separable extension of ${\mathbb F}_p(t)$ is completely different, and we show in \ref{curves}
that the analogue of the integer $d(g)$ for extensions $L$ of such $K$ does not exist when $p \neq 2$. 
We also discuss some evidence in \ref{evidence}, based on the 
Parity Conjecture, that the analogue of the integer $d(g)$ when $g=1$ might not exist.

Given a large degree extension $L$ of $K$, our inability to produce curves $X/K$ of small genus with new points over $L$ leads us to 
conjecture the following.

\begin{conjecture}
Given any $g \geq 2$, there exist a field $K$ and a separable extension $L/K$  such that 
no smooth projective geometrically connected curve $X/K$ of genus $g$ has a new point over $L$.
\end{conjecture}

We do not know of any finite separable extension $L/K$ of an infinite field $K$ such that
for a given $g \geq 1$,  
only finitely many smooth projective geometrically connected curves $X/K$ of genus $g$ have a new point over $L$.

In Section \ref{ex.cyc}, we consider some standard families of extensions $L/{\mathbb Q}$, such as the cyclotomic extensions ${\mathbb Q}(\zeta_{n})/{\mathbb Q}$,
and  present some examples of the smallest known positive integer $g$ for which there exists a curve $X/{\mathbb Q}$ of genus $g$ with a new point over $L$.
For instance, when $L={\mathbb Q}(\zeta_{31})$, Theorems \ref{pro.boundgeneral} and \ref{pro.bound} imply the existence of infinitely 
many curves $X/{\mathbb Q}$ of genus $g=6$ with a new point over $L$. In \ref{ex.special}, we produce an example of such a curve with $g=5$, and 
we indicate in \ref{rem.ellipticBSD} that the elliptic curve $50a1$ might have a new point over $L$. We do not know of any examples of curves $X/{\mathbb Q}$
with $g=2,3,4$ and with a new point over ${\mathbb Q}(\zeta_{31})$. 

For elliptic curves, the Birch and Swinnerton-Dyer conjecture
provides some computational evidence (see \ref{evidence}  and \ref{rem.ellipticBSD}) that   supports a positive answer to the following question: Given an extension $L/{\mathbb Q}$, does there always exist an elliptic curve $E/{\mathbb Q}$ such that $E$ has a new point over $L$?

\begin{convention}
\label{conv}  In this article, we call \emph{hyperelliptic curve over a field
$K$} any smooth projective geometrically connected curve   $X/K$ of 
positive genus $g \geq 1$, admitting a $K$-morphism of degree $2$ to $\mathbb P^1_K$. 
When $K$ has characteristic different from $2$, the function field $K(X)$ is thus $K$-isomorphic to $K(x)[y]/(y^2-f(x))$
for some $f(x)\in K[x]$ separable of degree at least $ 3$. 
Contrary to the usual convention in the literature, ours 
allows for a hyperelliptic curve to have genus $1$. This will simplify the statement of some of our theorems. 
\end{convention}

We thank Zev Klagsbrun and the referee for useful comments. Computations in this article were done using Magma \cite{Magma} and Sage \cite{Sage}.
\end{section}

 \begin{section}{Fixing the extension} \label{Fix}
 
Let $K$ be any field. In this section, we fix a finite extension
$L/K$, and ask whether there exist infinitely many smooth proper geometrically connected curves $X/K$ of a
given genus with a new point over $L$. 

When such a curve $X/K$ has a
new point $P$ over $L=K(P)$, 
the finite extension $L$ must be {\it simple} over $K$, that is, there
must exist $\beta \in L$ such that $L=K(\beta)$. This is clear if $L/K$ is separable, and we will treat in this article mostly this case.
Assume for the remainder of this paragraph that char$(K)=p>0$.
The vector space $\Omega_{K(P)/K}$ is a quotient of
$\Omega_{X/K}\otimes K(P)$ and, hence, when $P$ is smooth, has
dimension at most $1$ over $K(P)$.  The extension $L/K$ is simple
since  
$\dim_{K(P)} \Omega_{K(P)/K}$ is equal to the $p$-degree  $p$-$\deg(K(P)/K)$ 
(\cite{K}, 5.7, b)), and  when $p$-$\deg(L/K)> {\rm Trdeg}(L/K)$, the extension $L/K$ can be generated by $p$-$\deg(L/K)$ elements (\cite{K}, 5.11, b)).
When $\Omega_{L/K}=(0)$, the extension $L/K$ is separable.

From now on, unless specified otherwise, we  assume $L/K$ separable. 
It is clear that ${\mathbb P}^1_K$ has new points over any simple finite extension $L/K$. Thus in this article we will only be concerned with
 new $L$-points on curves of positive genus.
 Note that it was recently proved in \cite{MR}, 1.9, that if a smooth projective curve $X/K$ over a number field $K$  has a new point over all but finitely many finite extensions $L/K$, 
then $X$ is isomorphic over $K$ to ${\mathbb P}^1_K$. We assume in this section that ${\rm char}(K) \neq 2$, and we 
treat the case where ${\rm char}(K) = 2$ in  Section \ref{char2}.

Let $X/K$ be a hyperelliptic curve of genus $g$, and denote by $\pi: X \to {\mathbb
  P}_K^1$ a morphism of degree $2$. 
 We  say that $P \in X$ is a \emph{special point} if $\pi^*:
 K(\pi(P))\to K(P)$ is an isomorphism. If $X$ is given by an  
 equation $y^2=\ell(x)$ with $\ell(x)$ separable of degree $\ge 3$,
 and the closed point $P \in X$ corresponds to the orbit of a solution
 $(a,b)\in  \overline{K}^2$ of the equation $y^2=\ell(x)$, then $P$ is 
special if and only if $K(a)=K(a,b)$. 
If $g=1$, this definition depends 
on the choice of a double cover $\pi: X\to \mathbb P^1_K$. 
Note that $K$-rational points, and points on the ramification locus of $\pi$, are special. 
Let us now state the main theorem of this section.

\begin{theorem} \label{pro.boundgeneral} Let $K$ be any infinite field with ${\rm char}(K) \neq 2$. For each $i=1,\dots, t$, let $L_i/K$ be a  separable extension of degree $d_i$ (the extensions $L_i$ need not be distinct). Let $d:=\sum_{i=1}^t d_i $. Assume $d \geq 7$. Fix an integer $N>1$.
Then there exist infinitely many 
hyperelliptic curves $X/K$ 
(see Convention {\rm \ref{conv}}), pairwise non-isomorphic over $\overline{K}$, such that the following are true.
\begin{enumerate}[\rm (a)]
\item The curve $X$ contains distinct special new points $P_1\in X(L_1),
  \dots, P_t\in X(L_t)$. More precisely, let $Q_i$ denote the image of the point $P_i: \Spec L_i \to X$.
  Then $Q_1,\dots, Q_t$ are $t$ distinct closed special points of $X$.
\item Let $F$ be any finite extension of $K$ with $K$-linear field embeddings $L_i \hookrightarrow F$ for $i=1,\dots,t$. 
Then there exists a $K$-morphism $\varphi: X\to \Jac(X)$ of degree at most $4^g$ such that 
the image of at least one of the $P_i$'s in $\Jac(X)(F)$ 
has either order larger than $N$ or is of infinite order.
\item Write $d=4q+j$ for some $q \geq 1$ and $0\le j\le 3$. Then 
$X$ has genus  $  q-1 >0$ if $j\neq 3$, and  genus $  q$  if  $j=3$.
\item When $d=4q$ or $d$ is odd, then $X/K$ has at least one additional $K$-rational point distinct from $Q_1,\dots, Q_t$. 
When  $d=4q+3 $, $X/K$ has at least two additional $K$-rational points distinct from $Q_1,\dots, Q_t$.
\end{enumerate}
\end{theorem} 

The proof of Theorem \ref{pro.boundgeneral} is postponed to \ref{proof}. 
Our next corollary is a partial answer to the question raised at the beginning of this section.
Note that given a field $L/K$ of high degree $d$, 
our method
does not construct in general  a curve $X/K$ of low positive genus $g$ that contains a new point over $L$.

\begin{corollary} \label{thm.bound} Let $K$ be any infinite field with
  ${\rm char}(K) \neq 2$. Let $L/K$ be a  separable extension of
  degree $d$. 
\begin{enumerate}[\rm (1)]
\item If $d\le 9$, then there exists an elliptic curve $X/K$ with a
  new special point in $L$. 
\item If $d=10$, then there exists an elliptic curve $X/K$ with a
  new point in $L$. 
\item Let $g\ge 2$. 
Suppose that  
$$ 
g \geq 
\left\lbrace
\begin{matrix} 
\lfloor d/4 \rfloor  -1 &  \text{if } d\equiv   0, 1 \mod 4,\hfill   \\
\lfloor d/4 \rfloor \hfill & \text{if } d\equiv 2, 3 \mod 4.\hfill  
\end{matrix}
\right. 
$$
Then there exists a 
hyperelliptic curve $X/K$ of genus $g$ with a 
$K$-rational point and a new special point in $X(L)$. When $d\equiv 2
\mod 4$ and $g= \lfloor d/4 \rfloor -1 \geq 2$, there exists a hyperelliptic curve $X/K$ of genus $g$ with a new special point in $X(L)$.
\end{enumerate}
\noindent In all three cases, there are infinitely many such curves, pairwise 
non-isomorphic over $\overline{K}$. 
\end{corollary}

\begin{proof} Suppose $d\le 9$. Applying Theorem \ref{pro.boundgeneral} to 
the extensions $L_1:=L$ and $L_2=\cdots=L_{10-d}:=K$, we find infinitely many
hyperelliptic curves $X/K$ of genus $1$ with a $K$-rational point and
a new special point in $L$. This proves (1). 
Assertion (2) will be proved in Corollary \ref{cor.elliptic}. 

Let $g\ge 2$ as in (3). 
Let $s:=g-\lfloor d/4\rfloor +1$ if $d\equiv   0, 1 \mod 4$, and let $s:=g-\lfloor d/4\rfloor $ if $d\equiv   2,3 \mod 4$.
Let $t:= \max(0, 7-d)$. 
The desired curve $X/K$ is the curve obtained in Theorem \ref{pro.boundgeneral}
with the choice of extensions $L_1:=L$ when $s+t=0$, and $L_1:=L$ and $L_2=\dots=L_{4s+t}:=K$ when $4s+t>1$ (which always happens when $t>0$ since $g \geq 2$).
The fact that there are infinitely many such curves also follows immediately from \ref{pro.boundgeneral}.
\end{proof} 

\begin{remark} \label{special.points}
Let $K$ be a number field. Let $X/K$ be a smooth projective geometrically connected curve of genus $g$. Let $\pi:X \to {\mathbb P}_K^1$ 
be a finite $K$-morphism. 
Let $d\geq 1$ be an integer. 
Consider the set $S_{\pi, d}\subset X$ of closed points $P\in X$ of
degree $d$ over $K$ such
that the induced homomorphism on the residue fields $K(\pi(P))\to
K(P)$ is an isomorphism. 
Vojta shows in \cite{Voj}, Corollary 0.4, that if $g > (d-1)\deg(\pi) +1$, then $S_{\pi,d}$ is finite.
When $d=\deg(\pi)=2$, or when $d=1$, this result is due to Faltings, \cite{Fal}, Theorem 1.
When $X/K$ is hyperelliptic with $g \geq 2$
and $\pi$ is the canonical morphism of degree $2$, the set $S_{\pi,d}$ contains only special points by definition.

When $d \geq 2$,  we find using  Theorem \ref{pro.boundgeneral} a hyperelliptic curve of genus $g$ with $\lfloor (4g+6)/d \rfloor$ special closed points
of degree $d$. In case $d$ does not divide $4g+6$, we can also find such a curve with a $K$-rational point.

Returning to Vojta's Theorem, our construction implies the following:
{\it Let $K$ be a number field. Let $X/K$ be a 
hyperelliptic curve of genus $g$, with canonical morphism $\pi:X \to {\mathbb P}_K^1$ of degree $2$. If $g \geq 2d$, then $S_{\pi,d}$ is finite, and there exist examples of such a curve 
with at least eight special closed points of degree $d$.}
\end{remark}

The proof of Theorem \ref{pro.boundgeneral} uses the following crucial lemma, whose proof is well known, but which we reproduce here
due to its importance in defining the map $r : {\mathbb A}_K^d \to {\mathbb A}_K^{n+1}$ in \ref{def.g}. The polynomial $h(x)$ 
below is called the {\it approximate square root} of $m(x)$ in \cite{AM},
page 48. See for instance \cite{Mes} or \cite{Shi} for earlier uses of this lemma in the context of rational points on hyperelliptic curves. 

\begin{lemma} \label{lem.decomp} Let $R$ be a ring such that $2$ is invertible in $R$.
Let $m(x) \in R[x]$ be a monic polynomial of even degree $2n$. Then
there exists a unique pair of polynomials $h(x), \ell(x) \in R[x]$ such that 
$h(x)$ is monic of degree $n$, $\ell(x)$ has degree at most $n-1$, and 
$m(x) = h(x)^2-\ell(x)$.
\end{lemma}

\proof The uniqueness of such a pair is easy and is left to the reader.
The existence of $h(x)$ and $ \ell(x) \in R[x]$ is shown as follows. Write $h(x) = h_0+h_1x+ \dots + h_{n-1}x^{n-1}+x^n$
and $m(x) = m_0+m_1x+ \dots + m_{2n-1}x^{2n-1}+x^{2n}$.
Writing $h^2(x)$ explicitly, we find that for $i=0, \dots, n-1$, the coefficient $c_{2n-i}$ of $x^{2n-i}$ in $h(x)^2$ 
depends only on the $i$ coefficients $h_{n-i}, \dots, h_{n-1}$,
and moreover $$c_{2n-i}= 2h_{n-i}+ \text{\rm other terms},$$ where the `other terms' do not involve the variable $h_{n-i}$.
Hence, since $2$ is invertible, 
we can solve first the equation $c_{2n-1}=m_{2n-1}$ for $h_{n-1}$, and then $c_{2n-2}=m_{2n-2}$ for $h_{n-2}$, all the way to $h_{0}$.
With this choice of $h(x)$, we find that $m(x)-h(x)^2$ is a polynomial of degree at most $n-1$, and we set $-\ell(x):=m(x)-h(x)^2$.
\qed

\begin{emp} \label{def.f}
The proof of Theorem \ref{pro.boundgeneral} will also need the following facts. Let $L/K$ be a separable extension of degree $d$. Choose a generator $\alpha \in L$, so that we have $L=K(\alpha)$. 
Each element $\beta \in L$ can be written uniquely as $a_0 +a_1\alpha+\dots+a_{d-1}\alpha^{d-1}$
for some $a_0, \dots, a_{d-1} \in K$.  We denote by $m_\beta(x)$ the minimal polynomial of $\beta$ over $K$.

 Consider now the polynomial ring in $d$ variables $K[t_0,\dots,t_{d-1}]$, and its finite 
extension $K[t_0,\dots,t_{d-1}][\alpha]$. Let $\alpha_t := t_0+t_1 \alpha+\dots+t_{d-1}\alpha^{d-1}$.
The ring extension $K[t_0,\dots,t_{d-1}][\alpha]/K[t_0,\dots,t_{d-1}]$ is free of rank $d$, and we can thus consider the characteristic polynomial 
of $\alpha_t$, denoted by $\chi_{\alpha_t}(x) \in
K[t_0,\dots,t_{d-1}][x]$. 
Writing $\chi_{\alpha_t}(x) = P_0+P_1x+\dots+P_{d-1}x^{d-1}+x^d$ with 
$P_i \in K[t_0,\dots,t_{d-1}]$, 
we define $f_\alpha:\A^{d}_K \to \A^{d}_K$ to be given by the
homomorphism of 
$K$-algebras
\begin{equation} \label{mor}
K[T_0,\dots, T_{d-1}] \longrightarrow K[t_0,\dots,t_{d-1}], \quad T_i \longmapsto P_i.
\end{equation}
By construction, the evaluation of $f_\alpha$ at a $K$-rational point 
$(a_0,\dots,a_{d-1})$ sends the corresponding $\beta=a_0+a_1\alpha+\dots+a_{d-1}\alpha^{d-1}$ to its characteristic polynomial
over $K$.
When $\beta $ generates $L$, then 
its characteristic polynomial equals $m_\beta(x)$.
\end{emp}

\begin{proposition} \label{pro.finite} Let $K$ be any field, and let
  $L/K$ be a separable extension of degree $d$ with $L=K(\alpha)$. 
Then the morphism $f_\alpha$ defined by \eqref{mor} is finite and free of
degree $d!$. In particular $f_\alpha$ is surjective. 
\end{proposition}

\begin{proof} It suffices to prove that the extension $K[T_0,\dots, T_{d-1}] \longrightarrow K[t_0,\dots,t_{d-1}]$
is free of rank $d!$ after extension of the scalars to $\overline{K}$, that is, it suffices to prove
that $\overline{K}[T_0,\dots, T_{d-1}] \longrightarrow \overline{K}[t_0,\dots,t_{d-1}]$ is free of rank $d!$.

Let $\alpha_0:=\alpha$, and let $\alpha_0, \dots, \alpha_{d-1}$ denote the $d$ distinct conjugates of $\alpha$ in
$\overline{K}$. Let 
$$x_{i}:=\sum_{0\le j\le d-1}\alpha_i^j  t_j \in 
\overline{K}[t_0,\dots, t_{d-1}].$$ 
Let $N$ denote the $(d \times d)$ Vandermonde matrix with coefficients $n_{ij}= \alpha_i^j$, $0 \leq i,j \leq d-1$. 
As $L/K$ is separable, $N$ has non-zero determinant. Letting $N^t$ denote the transpose
of the matrix $N$, we have
$$(x_0, \dots, x_{d-1})= (t_0,\dots,t_{d-1})N^t.$$
Recall that by construction, 
$$\chi_{\alpha_t}(x)=P_0+P_1x+\dots+P_{d-1}x^{d-1}+x^d= \prod_{i=0}^{d-1}(x-x_i).$$ 
After we make the change of variables 
$(t_0,\dots,t_{d-1}) \mapsto (x_0, \dots, x_{d-1})$, 
$\overline{K}[T_0,\dots,T_{d-1}]$ 
is sent to the $\overline{K}$-subalgebra 
of $\overline{K}[x_0,\dots,x_{d-1}]$ generated by the symmetric
functions in $x_0,\dots,x_{d-1}$. 
It is a classical result  
that this extension is
Galois with Galois group the permutation group on $d$ elements. 
As $\overline{K}[T_0,\dots, T_{d-1}]\to \overline{K}[x_0,\dots,
x_{d-1}]$ is finite, it is injective because both domains have the same 
Krull dimension. 
\end{proof}
 
\begin{emp} \label{proof} 
{\it Proof of Theorem} \ref{pro.boundgeneral}. We start by defining three maps $f, \mu$, and $r$, and then use the composition $r \circ \mu \circ f$
in \ref{composition} to show that there exist some hyperelliptic curves which satisfy the conclusion (a) of the theorem. 
In \ref{orbits}, we refine our arguments and show that in fact infinitely such hyperelliptic curves exist. In \ref{proof-b}, we  set up some machinery involving Weil restrictions needed in the proof of Part (b) in \ref{N-tors}. Our constructions are such that the conclusions of (c) and (d) will be immediate to verify.

For each $i=1,\dots, t$, fix a generator $\alpha_i \in L_i$, so that $L_i=K(\alpha_i)$.
For later considerations, if $L_i=L_j$, then we choose $\alpha_i=\alpha_j$. 
For each $i$, consider the morphism $f_{\alpha_i} :\A^{d_i}_K \to \A^{d_i}_K$ defined in \eqref{mor}. 
Identify $\prod_{i=1}^t \A^{d_i}_K$ with $\A^{d}_K$, and consider 
$$ f:= f_{\alpha_1} \times \ldots \times f_{\alpha_t}: \A^{d}_K \longrightarrow \A^{d}_K.$$
Using \ref{pro.finite}, we find that $f$ is surjective. The source
scheme $\A^d_K$ will often be viewed as $\prod_{i=1}^{t} \Res_{L_i/K}
\A^1_{L_i}$ so that its $K$-rational points are canonically identified
with $\prod_i L_i$. 

For an affine space ${\mathbb A}^s_K$ which appears in the source or
target of the morphisms $\mu$ and $r$ that we define below, 
we will often think  of a
$\overline{K}$-rational point $(m_0,\dots,m_{s-1}) \in {\mathbb A}^s_K(\overline{K})$ as the monic polynomial 
$m_0+m_1 x+ \dots + m_{s-1}x^{s-1}+x^s\in \overline{K}[x]$. 
Define the affine morphism $$\mu: \A^{d}_K =\prod_{i=1}^{t} \A^{d_i}_K\longrightarrow \A^{d}_K$$
such that on $\overline{K}$-rational points, 
$$\mu(m_0^{(1)},\dots, m^{(1)}_{d_1-1}, \dots,  m_0^{(t)},\ldots, m^{(t)}_{d_t-1})
:= (m_0, m_1,\dots,m_{d-1}),$$ 
with the property that 
\begin{multline*} m_0+m_1x+ \dots+m_{d-1}x^{d-1}+x^d = \\ 
(m_0^{(1)}+m_1^{(1)}x + \dots + m^{(1)}_{d_1-1}x^{d_1-1}+x^{d_1}) \cdots  (m_0^{(t)}+ \dots +m^{(t)}_{d_t-1}x^{d_t-1} + x^{d_t}).
\end{multline*}
In terms of monic polynomials, the map $\mu$ is defined by 
$$(m^{(1)}(x),\dots, m^{(t)}(x)) \longmapsto m(x):=m^{(1)}(x)\cdots m^{(t)}(x).$$ 
We note that the morphism $\mu$ is surjective. Indeed, it is surjective on $\overline{K}$-points, since any monic polynomial $m(x) \in \overline{K}[x]$ 
of degree $d=\sum_{i=1}^t d_i$ can be factored over $\overline{K}$
into a product of monic polynomials of degrees $d_1, \dots, d_t$. When
$t=1$, $\mu$ is just the identity map. 
\end{emp} 

\begin{emp} \label{def.g}
Given $d\ge 7$, let us define 
$$ n:= \lfloor (d-1)/2 \rfloor.$$
Given any polynomial $m(x) \in K[x]$ of even degree $d$, we use Lemma \ref{lem.decomp}
to write $m(x)= h(x)^2 - \ell(x)$, with $\deg(h)=n+1$ and $\deg(\ell) \leq n$. 
Given any polynomial $m(x) \in K[x]$ of odd degree $d$, we use Lemma \ref{lem.decomp}
to write $xm(x)= h(x)^2 - \ell(x)$, with $\deg(h)=n+1$ and $\deg(\ell) \leq n$. 
Finally, we define a $K$-morphism 
$$r :\A^{d}_K \longrightarrow \A^{n+1}_K,$$
which sends $(m_0,\dots, m_{d-1}) $ to $(\ell_0,\dots, \ell_{n-1}, \ell_n)$, 
where $m(x)= m_0+ \dots+ m_{d-1}x^{d-1}+x^d$ and  $(\ell_0,\dots, \ell_{n-1}, \ell_n)$ are the coefficients of the associated polynomial $\ell(x)$
considered as a polynomial of degree $n$. We note here  
that the morphism $r$ is {\it surjective}. 
Indeed, 
it is clear when $d $ is even that this morphism is surjective on
$\overline{K}$-points,
since given any point in $\A^{n+1}_K(\overline{K})$ thought of as a polynomial  $\ell(x)$ of degree 
at most $n$, 
$\ell(x)$ is the image under $r$ of the point corresponding to $-\ell(x) $ thought of  as a polynomial of degree 
at most $d$.
Similarly, when $d $ is odd,   this morphism is surjective on $\overline{K}$-points 
since, given any $\ell(x)$ of degree 
at most $n$, 
$\ell(x)$ is the image under $r$ of the point corresponding to $m(x):=\frac{1}{x} \left((x^{(d+1)/2} +\sqrt{\ell(0)})^2-\ell(x)\right)$.
\end{emp}

\begin{emp} \label{composition}
 Consider the composition 
\begin{equation} \label{twomaps}
 \A^{d}_K \overset{f} \longrightarrow   \A^{d}_K \overset{\mu}
 \longrightarrow \A^{d}_K \overset{r}\longrightarrow \A^{n+1}_K.  
 \end{equation} 
By the previous considerations, we know that it is surjective. Given a $K$-rational point 
$$(a_0^{(1)},\ldots,a^{(1)}_{d_1-1}, \dots,
a_0^{(t)},\ldots,a^{(t)}_{d_t-1})$$
in the source of this composition, 
we obtain under the composition a $K$-rational point $(\ell_0,\dots, \ell_{n-1}, \ell_n)$ in $\A^{n+1}_K$
such that, letting $\ell(x)= \ell_0+\dots + \ell_{n}x^{n}$ and 
$$\beta_1:=a_0^{(1)}+\dots+a^{(1)}_{d_1-1}\alpha_1^{d_1-1}, \quad
\dots \quad, \beta_t:=a_0^{(t)}+\dots+a^{(t)}_{d_t-1}\alpha_t^{d_t-1},$$
 then the affine curve $y^2=\ell(x)$ contains 
the 
special points $(\beta_i, h(\beta_i))$ for $i=1,\dots, t$. 
Indeed, 
when $d$ is even and  
$\chi_{\beta_i}(x)$ denotes the characteristic polynomial of
$\beta_i$, then by construction 
$$\prod_{1\le i\le t} \chi_{\beta_i}(x) = h(x)^2-\ell(x).$$ 
Note that when $d$ is odd, 
$\ell(0)=h(0)^2$, so this curve also contains a $K$-rational point $(0, h(0))$.
If $d\equiv 0 \mod 4$, $\deg \ell(x)$ is odd and the point
at infinity is a rational point. 

Consider the (affine) dense open subset $V$ of $\A^{n+1}_K$ obtained as the complement of 
the union of the hyperplane $\ell_{n}=0$ and of the hypersurface ${\rm disc}(\ell(x)) =0$. Any $K$-rational point
in $V$ corresponds to a separable polynomial  $\ell(x) \in K[x]$ of degree
$n$. The equation $y^2=\ell(x)$ then defines a hyperelliptic curve
over $K$, of genus  
$g \ge 1$ if $d\ge 7$. 

We now define a dense open subscheme $U$ in the source space $\prod_{i=1}^t \A^{d_i}_K$ of the composition $r \circ \mu \circ f$ as follows.
Recall that since $L_i/K$ is finite and separable, there are only
finitely many proper extensions $F$ of $K$ in $L_i$. Since $F$ is a
subspace of the $K$-vector space $L_i$, if we identify $\A^{d_i}_K(K)$
with $L_i$ using the base $\{ 1, \alpha_i, \dots, \alpha_i^{d_i-1}\}$,
we can find a hyperplane $H_F$ in $\A^{d_i}_K$ such that 
$F\subseteq H_F(K)$. 
Consider the dense open subset $\Omega_i$ of $\A^{d_i}_K$ obtained as
the complement of the union of the hyperplanes $H_F$, 
the union being taken over all $F$ proper extensions of $K$ in $L_i$. 
Any $K$-rational point in $\Omega_i$ corresponds to $\beta_i \in L_i$
with $L_i=K(\beta_i)$.  Let $U_i$ denote the preimage of $\Omega_i$ under the projection map $\prod_{j=1}^t \A^{d_j}_K \to \A^{d_i}_K$.
Finally, let $U$ be the complement in $\cap_{i=1}^t U_i$ of the
$2$-diagonals of $\A^{d}_K$ 
(a $2$-diagonal in an affine space $\A^{s}_K= \Spec K[t_1,\dots,t_s]$ is a hyperplane defined by an equation of the form $t_i-t_j=0$ with $i\neq j$). Then any point in $U(K)$ induces pairwise
distinct generators $\beta_1, \dots, \beta_t$ of $L_1, \dots, L_t,$ respectively. 

Consider now $U'$, the preimage of the dense open subset $V \subseteq \A^{n+1}_K$ under the composition 
$r \circ \mu \circ f$. Then by construction, any $K$-rational point of 
\begin{equation}  \label{eq:Omega}
\Omega:=U\cap U' 
\end{equation}
defines a hyperelliptic curve, 
 of the form $y^2=\ell(x)$ with $\ell(x) $ of degree $n$,
that satisfies 
the conclusion (a) of the theorem. We now turn to  showing that we can find infinitely many non-isomorphic such curves.
\end{emp}

\begin{emp} \label{orbits} Identify $\A^{n+1}_K$ with ${\rm Spec}(K[\ell_0,\dots, \ell_n])$. Recall that over the open subscheme $V \subset \A^{n+1}_K$, the equation $y^2 = \ell_n x^n+\dots +\ell_0$ defines, by definition of $V$, 
a family of smooth proper hyperelliptic curves of genus $g=\lfloor (n-1)/2\rfloor >0$. We let 
$C\to V$ denote this family.

Suppose first $n\ge 5$, so that any fiber of $C/V$  has genus   $g \ge 2$.
Let ${\mathcal M_{g,K}}/K$  be the proper Deligne--Mumford stack 
of stable curves of genus
$g \geq 2 $ over $\Spec K$ (see \cite{D-M}, 5.1). This stack  admits a coarse moduli space $ M_{g,K}/K$ which is proper over $K$.

We will only use here the fact that $M_{g,K}/K$ exists, has positive dimension, and  
that the natural $K$-morphism 
$$ V\longrightarrow M_{g,K},$$
defined by the family $C/V$, has an image $W$ in $M_{g,K}$ of positive dimension. This latter fact is well known, 
and follows for instance from \ref{genus.small} or \ref{lem.semistable}.

Thus the composition 
$$ \Omega\longrightarrow  V \longrightarrow W $$ 
is dominant. As $\Omega(K)$ is dense in $\Omega\subseteq \A^{d}_K$,
its image in $W$ is a dense subset, which is infinite since  dim $W >0$.
Hence, we can find a subset $S$ of  $\Omega(K)$ which is infinite and maps bijectively onto its image.
The hyperelliptic curves associated to the elements of $S$ are pairwise non-isomorphic over $\overline{K}$, as desired.

Suppose now that $n=3$ or $4$, so that the curve $X/K$ defined by $y^2=\ell(x)$ has genus $1$. 
Consider the $K$-morphism 
$$j: V\longrightarrow \A^1_K,$$ 
which maps a separable polynomial $\ell(x)$ of degree $n$ to the $j$-invariant
of the Jacobian of $X$. Since we could not find an adequate reference in the literature, 
we briefly recall in \ref{orbit3} how this morphism of schemes is defined.
This morphism is surjective on $\overline{K}$-points and the same arguments 
as above show that we can find a subset $S$ of  $\Omega(K)$ which is infinite and maps bijectively onto its image.
The hyperelliptic curves associated to the elements of $S$ are pairwise non-isomorphic over $\overline{K}$, as desired.
\end{emp}

\begin{emp}\label{proof-b}
It remains to show that infinitely many curves satisfy both conclusions (a) and  (b)
of the theorem. For this we need to introduce the following notation.
Let $R:=\prod_{i=1}^t L_i$.  
We write the base change $V\times_{\Spec K} \Spec R$ as $V_R$. The tautological family $C \to V$ 
induces by base change the family $C_{V_R} \to V_R$, and we can
consider the Weil restriction ${\rm Res}_{V_R/V} \, C_{V_R}  \to V$ along 
the finite free morphism $V_R \to V$, which exists since $C_{V_R} \to V_R$ has projective fibers
 (\cite{BLR}, 7.6, Theorem 4). 
One checks that there is a natural $V$-isomorphism  
\begin{equation} \label{eq.Weilres}
{\Res}_{V_R/V} \,  C_{V_R} \longrightarrow \prod_i \Res_{V_{L_i}/V} \, C_{V_{L_i}}.
\end{equation}

The composition of maps 
$\A^{d}_K \overset{f} \longrightarrow   \A^{d}_K \overset{\mu}
\longrightarrow \A^{d}_K \overset{r}\longrightarrow \A^{n+1}_K $
introduced in \eqref{twomaps}  
induces by restriction a natural morphism $\Omega\to V$ as in \eqref{eq:Omega}.
This latter morphism is loosely described as follows (when $d$ is even). 
Given $\beta =(\beta_1,\dots, \beta_t)$
corresponding to a point in $\Omega(K)$, write the product of the
minimal polynomials of $\beta_1,\dots, \beta_t, $   as $h_{\beta}(x)^2
- \ell_\beta(x)$. Then the image of $\beta $ is  $\ell_\beta(x)\in V(K)$.
The curve $y^2=\ell_\beta(x)$ comes equipped with the points $(\beta_i, h_{\beta}(\beta_i))$.
Our goal now is to express precisely how these points depend on $\beta$.
This is done through the following claim: {\it The  morphism $\Omega\to V$ factors as
$$\Omega \stackrel{\iota}{\longrightarrow} {\Res}_{V_R/V} \,  C_{V_R}
\longrightarrow V,$$
with an {\it injective} morphism $\iota$.} 
In an imprecise manner, this factorization 
can be described as follows: Given $\beta =(\beta_1,\dots, \beta_t)$
corresponding to a point in $\Omega(K)$, 
then  
$$
\iota(\beta)=(\beta, h_\beta(\beta_1),\dots, h_\beta(\beta_t), \ell_\beta(x)).$$
The second map, ${\Res}_{V_R/V} C_{V_R} \to V$, is simply the projection onto $\ell_\beta(x)$. 
To define $\iota$ more rigorously, we proceed as follows. 
\end{emp}

\begin{emp} An affine open subset of the curve $C/V$ is given in the affine plane $\A^2_V$
by the equation $y^2 = \ell_n x^n+\dots +\ell_0$.
This  affine scheme is the spectrum of the ring 
$$\cO_V(V)[x,y]/(y^2-(\ell_n x^n+\dots +\ell_0)).$$
In order to ease the notation, we identify 
as in \eqref{eq.Weilres}
${\Res}_{V_R/V} \, C_{V_R} $ with $ \prod_i \Res_{V_{L_i}/V} \, C_{V_{L_i}}$,
and
only describe here the morphism 
$\iota^{(1)}: \Omega \to {\rm Res}_{V_{L_1}/V} \, C_{V_{L_1}}$. 
Recall that $\alpha_1$ is the generator of $L_1/K$ 
involved in the
definition of the morphism $f$ (see \ref{proof}), so that we use below when computing the Weil restriction that
$\{1, \alpha_1, \dots, \alpha_1^{d_1-1}\}$ is a basis for $L_1$ over $K$. Denote by
$(a_0^{(1)}, \dots, a_{d_1-1}^{(1)}, \dots, a_0^{(t)}, \dots,
a_{d_t-1}^{(t)})$ the variables of $\A^d_K=\prod_i \A^{d_i}_K$. Let us  
introduce variables $y_0,\dots, y_{d_1-1}$  
and define polynomials 
$$f_i \in K[\ell_0, \dots, \ell_n, a_0^{(1)},\dots,a_{d_1-1}^{(1)},y_0,\dots,
y_{d_1-1}], 
\quad i=0,\dots, d_1-1,$$
uniquely by the equation 
$$\sum_{i=0}^{d_1-1}f_i \alpha_1^i=\left(\sum_{i=0}^{d_1-1}
  y_i\alpha_1^i\right)^2 - 
\left(\ell_n  \left(\sum_{i=0}^{d_1-1}
    a_i^{(1)}\alpha_1^i\right)^n+\dots +\ell_1  \left(\sum_{i=1}^{d_1-1} a_i^{(1)}\alpha_1^i\right)+\ell_0 \right).$$
The Weil restriction from $V_{L_1}$ to $V$ of an open subscheme of $C_{V_{L_1}}$ can be given as the closed subscheme of $\A^{d_1}_V \times_V \A^{d_1}_V$ 
whose affine algebra is
$$A^{(1)}:=\cO_V(V)[a_0^{(1)},\dots,a_{d_1-1}^{(1)},y_0,\dots, y_{d_1-1}]/(f_0,\dots, f_{d_1-1}).$$
Next we describe 
the homomorphism of $K$-algebras $\iota_1^* : A^{(1)} \to
\cO_\Omega(\Omega)$ such that
$\iota^{(1)}: \Omega \longrightarrow {\rm Res}_{V_{L_1}/V} \, C_{V_{L_1}}$ is obtained as the composition of
$\iota_1: \Omega \to \Spec A^{(1)}$ with the open inclusion  $\Spec A^{(1)} \subseteq {\rm Res}_{V_{L_1}/V} \, C_{V_{L_1}}$. 
Recall that the morphism 
$\A_K^d \overset{r} \longrightarrow  \A_K^{n+1}$
corresponds to a homomorphism
$$r^* : K[\ell_0,\dots, \ell_n] \longrightarrow K[m_0,\dots, m_{d-1}].$$
In the case $d$ even, polynomials $h_0,\dots,h_{n}\in K[m_0,\dots, m_{d-1}]$ 
are defined by the relation 
$$m_0+m_1x+\dots+m_{d-1}x^{d-1}+x^d= (h_0+\dots+h_{n}x^n+x^{n+1})^2-
(r^*(\ell_0)+\dots+r^*(\ell_n) x^n).$$
The multiplication morphism
${\mathbb A}_K^d \overset{\mu } \longrightarrow {\mathbb A}_K^d$
corresponds to a homomorphism
$$ \mu^* : K[m_0,\dots, m_{d-1}]\longrightarrow K[m_0^{(1)},\dots,m_{d_1-1}^{(1)}, \dots, m_0^{(t)},\dots,m_{d_t-1}^{(t)}].$$
Define polynomials $Y_0, \dots, Y_{d_1-1} \in K[m_0^{(1)},\dots,m_{d_1-1}^{(1)}, \dots, m_0^{(t)},\dots,m_{d_t-1}^{(t)}]$
such that $Y_0+Y_1x + \dots + Y_{d_1-1}x^{d_1-1}$ is the remainder of the Euclidean division of the image of $h_0+\dots+h_{n}x^n+x^{n+1}$
by the polynomial $m_0^{(1)}+\dots+m_{d_1-1}^{(1)}x^{d_1-1} +x^{d_1}$.

The morphism
${\mathbb A}_K^d \overset{f } \longrightarrow {\mathbb A}_K^d$
corresponds to a homomorphism
$$ f^* :  K[m_0^{(1)},\dots,m_{d_1-1}^{(1)}, \dots,
m_0^{(t)},\dots,m_{d_t-1}^{(t)}] \longrightarrow
K[a_0^{(1)},\dots,a_{d_1-1}^{(1)}, \dots, a_0^{(t)},\dots, a_{d_t-1}^{(t)}].$$
The homomorphism $\iota_1^*$ is then defined as follows:
$$
\begin{array}{rcl}
\iota_1^*(z)&:= & (r\circ \mu \circ f)^*(z), \ {\rm for \ all \ } z\in \cO_V(V),\\
\iota_1^*(a_i^{(1)})&:= & a_i^{(1)}, \\
\iota_1^*(y_i)&:= &  f^*(Y_i).\\
\end{array}
$$

Note that the canonical homomorphism $K[a_0^{(1)}, \dots,
a_{d_1-1}^{(1)}]\to A^{(1)}$ induces a morphism $\Spec A^{(1)}\to \A^{d_1}_K$
whose pre-composition with $\iota_1: \Omega \to \Spec A^{(1)}$ is nothing but the projection 
$\Omega\subset\A^d_K\to \A^{d_1}_K$. This implies that 
$\iota:=\iota^{(1)}\times \dots \times \iota^{(t)}: \Omega\to \Res_{V_R/V}C_{V_R}$ is injective (it is actually an
immersion, but we do not need this fact).
\end{emp}

\begin{emp} \label{N-tors}
Now that the factorization $\Omega \longrightarrow {\rm Res}_{V_R/V} \, C_{V_R}  \longrightarrow V$ has been defined, 
we proceed as follows. 
The tautological family $C \to V$ of hyperelliptic curves introduced in \ref{orbits}  
 comes equipped with a natural $V$-cover $C \to {\mathbb P}^1_V$ of degree $2$.
The choice of the coordinate function $x$ on ${\mathbb P}^1_V$ determines a section $\infty: V \to {\mathbb P}^1_V$, 
and we denote by $D$ the divisor on $C$ obtained as the pull-back of the section $\infty$ under the double cover $C \to {\mathbb P}^1_V$.

Let $J/V$ denote the Jacobian of $C/V$, and 
denote by $J[N]/V$ the kernel of the multiplication-by-$N$ on
$J$. Then $J[N]$ is a closed subscheme of $J$, finite over
$V$. The   divisor $D$ on $C/V$ 
induces a finite $V$-morphism 
$\varphi_D : C\to J$ of degree at most $2^{2g}$. 
The natural morphism $ C_{V_R} \to J_{V_R}$ obtained by base change
from $\varphi_D$  induces a finite $V$-morphism 
$$\varphi:  {\rm Res}_{V_R/V} \, C_{V_R} \longrightarrow  {\rm Res}_{V_R/V} \, J_{V_R},$$
and  ${\rm Res}_{V_R/V} \, J_{V_R}$ contains as a closed subscheme the scheme  $F_N:={\rm Res}_{V_R/V} \, J[N]_{V_R}$.
Consider the composition 
$$\Omega \longrightarrow {\rm Res}_{V_R/V} \, C_{V_R} \longrightarrow  {\rm Res}_{V_R/V} \, J_{V_R}.$$

{\it We claim that the the image of $\Omega$ in ${\rm Res}_{V_R/V} \, J_{V_R}$
does not lie in the closed subset $F_N$.} Indeed, assume {\it ab absurdo} that it does. Then, since $\Omega  \longrightarrow
{\rm Res}_{V_R/V} \, C_{V_R}$ is injective 
and since $F_N \to V$ is  a finite morphism, we find that the composition 
$$\Omega \longrightarrow {\rm Res}_{V_R/V} \, C_{V_R} \longrightarrow {\rm Res}_{V_R/V} \, J_{V_R} \longrightarrow V$$
is quasi-finite. Since $\Omega \subseteq {\mathbb A}_K^{d}$ and $V \subseteq {\mathbb A}_K^{n+1}$ are two schemes of unequal dimension
and the morphism $\Omega \to V$ is dominant, we obtained a contradiction. Thus, the image of $\Omega$ is not contained in $F_N$.

Let $U_N$ denote the  complement of 
$F_{N!}$ in ${\rm Res}_{V_R/V}  \, J_{V_R}$. Let $U''$ denote the preimage in $\Omega$ of the open subscheme $U_N$.
The open subscheme $U''$ is  dense in $\Omega$, and we can apply the argument of \ref{orbits} to $U''$ instead of $\Omega$.
Hence, we can find a subset $S$ of  $U''(K)$ which is infinite and maps bijectively onto its image.
The hyperelliptic curves associated to the elements of $S$ are pairwise non-isomorphic over $\overline{K}$, as desired. 
Each such hyperelliptic curve $X/K$ has a new point $P_i \in X(L_i)$  for $i=1,\dots, t$. In addition,
since $S \subset U''(K)$, the image of at least one $P_i$ under the natural $L_i$-morphism $X_{L_i} \to {\rm Jac}(X_{L_i})$ given by the divisor of degree $2$ `at infinity' cannot be a torsion point of order dividing $N!$. Thus the order of this image is bigger than $N$, or the image has infinite order.
This completes the  proof of Theorem \ref{pro.boundgeneral}, since Parts (c) and (d)  follow immediately from our construction.
\qed
\end{emp}
 
\begin{remark} \label{rem.weakermethod}
Let $K$ be a field with $\mathrm{char}(K)\neq 2$.  Let
  $L/K$ be a finite separable extension of degree $d$.
When $K$ is infinite, or is finite with $|K|$ large enough, we show below 
the existence of a curve $X/K$ of low genus $g$ with a new point over $L$ and a $K$-rational point. The method used here is much more elementary
than the method  in Theorem \ref{pro.boundgeneral}, and also has the advantage of being completely explicit, but the bound on $g$ obtained with this method is slightly weaker
than the bound in Theorem \ref{pro.boundgeneral}.  

Choose an element $\alpha \in L$ with $L=K(\alpha)$. 
Let $m(x)\in K[x]$ denote the minimal (monic) polynomial of $\alpha$ over $K$.
Let $f(x) \in K[x]$ be a 
monic non-zero multiple of $m(x)$ of even degree $2n \geq d$ such that $f(x)$ is not a square in $K[x]$
and such that $f'(x) \neq 0$ (e.g., $f(x)=x^{2n-d}m(x)$). 
Lemma \ref{lem.decomp} lets us  find $h(x) \in K[x]$ monic of degree $n$ such that 
 $ \ell(x):= h(x)^2 -f(x) $ is of degree at most $n-1$. The curve defined by $y^2=\ell(x)$ has the new point $(\alpha, h(\alpha))$, and if it has positive genus,  there is nothing to do. 

Assume then that $y^2=\ell(x)$ defines a curve of genus $0$.
 For any $ c \in K^*$, 
consider the possibly singular curve defined by  the equation
 $$y^2 = \ell(x) + 2  h(x)c + c^2.$$
 By construction, this curve has a new point in $X(L)$, namely $(\alpha, h(\alpha)+c)$.
It is possible to show using only elementary arguments that there exists $c \in K^*$ such that $\ell(x) + 2  h(x)c + c^2$ has only simple roots in $\overline{K}$, 
 and such that $2c$,  the leading term of $\ell(x) + 2  h(x)c + c^2$, is a square in $K$. Then the projective curve $X/K$ associated
 with the above affine equation is a smooth hyperelliptic curve of genus $g = \lfloor(n-1)/2 \rfloor$,
and it has a $K$-rational point at infinity.  (This bound for $g$ is not as good as the bound achieved in Theorem \ref{pro.boundgeneral}
when ${\rm deg}(h(x))$ is odd.) We leave the details to the reader.
\end{remark}

\end{section}

\begin{section}{A refinement} \label{sec.refinement}

When the extension $L/K$ is of a special form, it is possible to further refine some of the bounds given in Theorem \ref{pro.boundgeneral}. 
 
\begin{proposition} \label{pro.bound} Let $K$ be any infinite field
  with $\mathrm{char}(K) \neq 2$. Let $L/K$ be a  separable extension
  of degree $d:=ke$, 
with $d \geq 9$, $k \geq 2$, and $e \geq 3$.
Assume that $L$ contains a subextension $L_0/K$ of degree $e$ over $K$
such that $L=L_0(\sqrt[k]{\beta_0})$ for some $\beta_0 \in L_0^*$
and some root $\sqrt[k]{\beta_0} \in L$ of $x^k-\beta_0$. Fix an integer $N>1$.
Then there exist hyperelliptic curves $X/K$ 
such that the following are true.
\begin{enumerate}[\rm (a)]
\item  $X/K$ is given by an affine equation   $y^2=\ell(x^k)$ with
$\deg \ell(x^k)  = (d-k)/2$   
if $e$ is odd, and
$\deg \ell(x^k)  =  (d-2k)/2$ 
if $e$ is even. In addition,  $\ell(x^k)$ is separable.
\item  $X(L)$ contains a special new point $P$, and if 
$e$ is odd, $X$ also has a $K$-rational point.
\item There exists a finite $K$-morphism $\varphi : X\to \Jac(X)$ of degree at most $2^{2g}$
 such that 
the image under $\varphi$ of each of the $d$ conjugates of $P$ 
is of order larger than $N$ or has infinite order in ${\rm Jac}(X)(L)$. 
\end{enumerate}
Moreover, if 
$e\geq 5$, then there exist infinitely many such
curves, pairwise non-isomorphic over $\overline{K}$.
\end{proposition} 
 
\begin{proof} 
Let  $m_0(x)\in K[x]$ denote the minimal polynomial of $\beta_0$ over $K$.
The minimal polynomial of $\sqrt[k]{\beta_0}$ over $K$ is $m_0(x^k)$. 
Let $t(x):=x$ if $e=\deg m_0(x)$ is odd and $t(x):=1$ otherwise. 
Lemma~\ref{lem.decomp} lets us   write
$$t(x)m_0(x)=h(x)^2-\ell(x)$$ 
for some appropriate $h(x),\ell(x) \in K[x]$.
The curve $X/K$ defined by the affine equation 
$y^2=\ell(x^k)$ 
has the new $L$-point $(\sqrt[k]{\beta_0},h(\beta_0))$ and, when $e$ is
odd, it also has the $K$-rational point $(0,h(0))$. 

When $p={\rm char}(K)>0$, our hypothesis that $L/L_0$ is separable implies that $p$ does not divide $k$.
Thus, for $\ell(x^k)$ to have positive degree and distinct roots, 
we only need $\ell(x)$ to have positive degree and distinct non-zero roots.  The polynomial $\ell(x)$ depends on the choice of $\beta_0 \in L_0$, 
and it may happen that $y^2=\ell(x)$ has genus $0$. 
When this is the case, we can pick $\gamma \in L_0^*$, and let $\beta'_0:=\beta_0\gamma^k$. 
Define $\sqrt[k]{\beta_0'}:=\gamma\sqrt[k]{\beta_0}$, with $\sqrt[k]{\beta_0'}$ a root of $x^k-\beta_0'$. Then we still have  $L=L_0(\sqrt[k]{\beta_0'})$,
and we can consider as above the new curve $y^2=\ell(x)$ associated with $\beta_0'$.

We use below this simple observation to show that there is one (and infinitely many  when $e \geq 5$) curve $X/K$ which satisfy conditions (a) and (b) of the proposition.
 
We start by considering the morphism 
$\A^1_{L_0}\longrightarrow \A^1_{L_0}$, with $\gamma\mapsto \gamma^k \beta_0$,  and denote its
Weil restriction to $K$ by 
$$ s :  {\mathbb A}^e_K \longrightarrow {\mathbb A}^e_K.$$ 
On $K$-rational points, $s$ maps 
$(x_0,\ldots,x_{e-1}) $ to $(y_0,\ldots,y_{e-1})$, 
where 
$$y_0+y_1\beta_0+\dots+y_{e-1}\beta_0^{e-1}=
(x_0+x_1\beta_0+\dots+x_{e-1}\beta_0^{e-1})^k\beta_0.$$
This is a finite surjective morphism.  
We obtain a composition
$$\A^{e}_K \overset{s} \longrightarrow  \A^{e}_K  \overset{f_{\beta_0}} \longrightarrow  \A^{e}_K \overset{r}\longrightarrow \A^{n+1}_K, $$
where $n:=\deg(t(x)m_0(x))/2-1\in \mathbb N$, and the last two maps are
described in \ref{def.f} and \ref{def.g}, respectively. This
composition is surjective. On $K$-rational points, it maps 
$\gamma\in L_0$ to $\ell(x)\in K[x]$ such that the characteristic 
polynomial of $\gamma^k \beta_0$ over $K$ is equal to 
$h(x)^2-\ell(x)$ with appropriate $h(x)$ and $\ell(x)$ as in Lemma~\ref{lem.decomp}.  

Our argument now is very similar to the arguments given in the proof of Theorem \ref{pro.boundgeneral}. 
We will only sketch the main steps below, leaving the detailed verification to the reader. 
Let $V$ denote the open subset of the target space $\A^{n+1}_K$
corresponding to the polynomials $\ell(x)$  such that $\ell(x)$ has
degree $n$ and has distinct non-zero roots.
Let $U$ denote the open subset of $\A^{e}_K=\Res_{L_0/K}\A^1_{L_0}$ corresponding to the elements $\delta \in L_0$ with $L_0=K(\delta)$.
Let $W\subseteq \A^{e}_K $ be the intersection of 
$s^{-1}(U)$ with the preimage of $V$ by the above composition $r\circ f_{\beta_0} \circ s$. 
Then $W$ is open and non-empty, and any rational point 
$\gamma\in W(K)\subseteq L_0$ corresponds to an $\ell(x)$ of degree $n$ with 
distinct non-zero roots in $\overline{K}$. As $e\ge 3$, 
any such $\gamma$ gives rise to an affine equation 
$y^2=\ell(x^k)$ defining a hyperelliptic curve satisfying 
properties (a) and (b). 

To prove Part (c), fix an $N\ge 2$. Let $\A^{n+1}_K=\Spec K[\ell_0, \dots, \ell_n]$. Consider the universal
hyperelliptic curve $C$ over $ V\subseteq \A^{n+1}_K$ defined by the equation 
$$ y^2=\ell_nx^{nk}+\ell_{n-1}x^{(n-1)k}+\cdots + \ell_1 x^k + \ell_0.$$
We can restrict $W$ further 
as in \ref{N-tors}, 
so that that the images by $\varphi_D$ of the 
$K$-conjugates of the point 
$P=(\gamma \sqrt[k]{\beta_0}, h(\gamma^k\beta_0))$ have 
order larger than $N$ for $\gamma\in W(K)$. 

It remains to prove that our family of hyperelliptic curves is
infinite when $e\ge 5$. This amounts to showing that the natural 
morphisms $V\to M_{g,K}$ (if $g\ge 2$) and $j : V\to \A^1_K$ (taking $j$-invariant)
when $g=1$ (see \ref{orbits}), are non-constant. This is equivalent
to saying that over $\overline{K}$, for any given $n, k\ge 2$,  $k$ prime to
$\mathrm{char}(K)$, the family of hyperelliptic curves defined by 
$$ y^2=a_{n}x^{nk}+\cdots + a_1 x^k + a_0, \quad 
(a_0, \dots, a_n)\in V(\overline{K})$$ 
is non-constant. This statement follows from Lemma \ref{lem.semistable}, 
which we apply to the subfamily $ y^2= x^{nk}  + \lambda x^k + 1$
if $n-1$ is prime to ${\rm char}(K)$, and to the subfamily $ y^2= x^{nk}  + \lambda x^{2k} + 1$
when ${\rm char}(K)$ divides $n-1$.
\end{proof}

\begin{remark} When $e=3$ or $4$, we have $\deg \ell(x)=1$ and  the curves appearing in  Proposition \ref{pro.bound} (a) are of the form $y^2=ax^k+b$ and are thus
 isomorphic over $\overline{K}$. 
\end{remark}

\begin{corollary} \label{cor.elliptic} 
Let $K$ be an infinite field with $\text{\rm char}(K) \neq 2$. 
Let $L/K$ be a separable extension of degree $d $, with $ 1 \leq d
\leq 10$.
Fix an integer $N\geq 2$. 
\begin{enumerate}[{\rm (1)}] 
\item 
There exist infinitely many elliptic curves $X/K$  with distinct
$j$-invariants and with a new point in $X(L)$, which
is either of order bigger than $N$  or of infinite order in $X(L)$.
\item 
When $K$ is a global field, these new points can be chosen to be of infinite
order.
\end{enumerate}
\end{corollary}

\begin{proof} (1) The corollary follows from Theorem \ref{pro.boundgeneral} when $d \leq
9$ and $d\ne 5, 6$. 
Indeed, for $d=1$, use  $L_i=K$ for $i=1,\dots, 8$, so that $\sum_i [L_i:K]=8$.
For $d=2$, use $L_i=L$ for $i=1,\dots, 4$, so that $\sum_i [L_i:K]=8$. 
For  $d=3$, use $L_i=L$ for $i=1,2,3 $, so that $\sum_i [L_i:K]=9$.
For  $d=4$, use $L_i=L$ for $i=1,2 $, so that $\sum_i [L_i:K]=8$.
For  $d=7,8,9$, use $L_1=L$.

For the case $d=5$, we use  $L_i=L$ for $i=1,2 $, so that $\sum_i
[L_i:K]=10$,  
and consider the infinitely many curves $X/K$ of genus $1$ given by 
Theorem \ref{pro.boundgeneral}. If the curve is elliptic, it thus has a point of degree $5$.
Assume then that the curve does not have a $K$-rational point.
Each such curve has then two new points over $L$ such
that the image  of one of the two points 
under a $K$-morphism $X \to {\rm Jac}(X)$ of degree $4$ has order bigger
than $N$ or has infinite order. 
Since the morphism is of degree $4$ and the point has degree $5$,
we find that  the image  in the Jacobian  of a point of degree $5$ in $X$
still has degree $5$ in the Jacobian. As the curves $X$ are pairwise non-isomorphic
over $\overline{K}$, the same is true for their Jacobians $\Jac(X)$.

Assume now that $d=10$ and that $L$ does not have a subextension of degree $5$
over $K$. Consider $X/K$ of genus $1$ with a new point over $L$ and the $K$-morphism $X\to \Jac(X)$ of degree $4$ 
provided by Theorem \ref{pro.boundgeneral}. The image of any new point of degree $10$ has degree $10$ or $5$ over $K$.
By assumption on $L$, the second case is excluded and therefore we obtain a new point of degree $10$ in 
$\Jac(X)$.

When $d=10$ and  $L/K$ has a subextension $L_0$ of degree $5$,
we apply  Proposition \ref{pro.bound}. In this case we obtain infinitely many curves of the form $y^2=ax^4+bx^2+c$ with a $K$-rational point, 
a new point $P$ over $L$, 
and a morphism of degree $4$ to their Jacobians. Each curve also has a $K$-rational point $Q=(0,\sqrt{\ell(0)})$, and we endow $X$ with the group structure that has $Q$ as origin.
We consider the divisor $D:=Q$, so that the morphism $\varphi_D: X \to {\rm Pic}^0(X)$ is a morphism of elliptic curves (it sends $Q$ to the origin of ${\rm Pic}^0(X)$).
In order to have the point $P$ have order at least $N$ on the elliptic curve $(X,Q)$, we apply Proposition \ref{pro.bound} in the case where the points obtained in the Jacobians
have order at least $4N$.

Assume now that $d=6$.  Theorem \ref{pro.boundgeneral} (applied in the case where $L_1=L$ and $L_2=K$, and with the integer $(N!)^2$) shows that there exist infinitely many elliptic curves $X/K$ with a new point $P \in X(L)$ 
and a   point $Q \in X(K)$ such that the order of either $P$ or $Q$ is bigger than $(N!)^2$. Indeed, Theorem \ref{pro.boundgeneral} produces curves of genus $1$ with such points $P$ and $Q$, and we endow these curves, given by an equation $y^2=\ell(x)$, with the structure of elliptic curves by picking the point at infinity as the origin of the group law. Then the induced morphism $\varphi: X \to {\rm Jac}(X)$ is a morphism of group schemes, and if the image of either 
$P$ or $Q$ in ${\rm Jac}(X)(L)$ has order at least $(N!)^2$ or is of infinite order, then the same holds for $P$ or $Q$.
If the order of $P$ is at most $N$, then consider the point $P+Q$. It is clear that $P+Q$ is also a new point in $X(L)$ since if it were defined over a smaller extension $M \subset L$, then the point $P=(P+Q)-Q$ would also be defined over $M$, a contradiction. If the order of $P+Q$ were also at most $N$, then the
order of $Q=(P+Q)-P$ would   divide the product of the orders of $P$ and $P+Q$, and thus would be a divisor of $(N!)^2$. This would be a contradiction since when the order $P$ is at most $N$, the order of $Q$ is larger than $(N!)^2$.

(2) Assume that $K$ is a number field. Then  Merel \cite{Mer} shows that
there exists an integer $N$, depending only on $L$, such that if $X/L$ is any elliptic curve
with an $L$-rational torsion $P$, then the order of $P$ is at most
$N$. It follows from 
the first statement of the corollary that there exist
infinitely many elliptic curves $E/K$ with a new point $P$ over $L$ of 
order bigger than $N$ or of infinite order. Hence, such $P$ has infinite order.

Assume now that  $K/k$ is the function field of a smooth 
connected curve over a finite field $k$.
Then  Levin \cite{Lev}, Theorem 1, shows that there exists an integer $N$ such that if $X/L$ is any elliptic curve
with an $L$-rational torsion $P$ and whose $j$-invariant $j(X) $ is not algebraic over $k$, then the order of $P$ is at most $N$.
Note that in our field $L$, there are only finitely many elements which are algebraic over $k$ by hypothesis. 
Therefore, any infinite family of elliptic curves over $L$ with distinct $j$-invariants has an infinite subfamily with $j$-invariants transcendental over $K$.
It follows from the first statement   
of the corollary  that there exist
infinitely many elliptic curves $E/K$ with a new point $P$ over $L$ of
order 
bigger than $N$ or of infinite order.  Hence, such $P$ has infinite order.
\end{proof}
\begin{remark}
One may wonder whether it is possible to strengthen Theorem \ref{pro.boundgeneral} and Corollary \ref{cor.elliptic}, over 
number fields for instance, 
to show that given a separable extension $L/K$ of degree at most $10$, there exist infinitely many elliptic curves $X/K$ 
with a new point $P$ in $X(L)$ {\it and  trivial torsion subgroup} in $X(K)$.
\end{remark}

\begin{corollary} \label{cor.infinitetelymanycurves} Let $K$ be a number field.
For each odd integer $g \geq 1$, there exists a separable extension $L/K$ with $[L:K] = 6(g+1)$ such that there exist
infinitely many smooth projective geometrically connected curves $X/K$ of genus $g$, pairwise non-isomorphic over $\overline{K}$ and with a new point over $L$.
\end{corollary}
\proof
 Let $L_0=K(\beta)$ be an extension of degree $6$. Let $L:=L_0(\sqrt[g+1]{\beta})$ and assume that $[L:L_0]=g+1$.
Apply Proposition  \ref{pro.bound} to the extension $L$ and its subextension $L_0/K$
to find infinitely many pairwise non-isomorphic hyperelliptic curves $X/K$ of genus $g$ of the form $y^2=ax^{2g+2}+bx^{g+1}+c$ with a new point over $L$. \qed 

\smallskip
The hypothesis in \ref{cor.infinitetelymanycurves} that $K$ is a number field is only used to justify that $K$ indeed has extensions of the form $L$ and $L_0$.
When ${\rm char}(K)>2$, Corollary \ref{cor.infinitetelymanycurves} is strengthened in \ref{curves}. It is also strengthened below when $g=1$.

\begin{proposition} \label{cor.extension15-16} 
Let $K$ be an infinite field. When ${\rm char}(K) \neq 2$, assume that 
$L/K$ is an abelian extension of degree $12$, $14$, $15$, $20$, $21$, or $30$.
When ${\rm char}(K) = 2$, assume that 
$L/K$ is an abelian extension of degree $12$, $14$, $15$, or $20$.
Then there exist infinitely many elliptic curves $E/K$ with pairwise distinct $j$-invariants and  with a new point over $L$.
\end{proposition}
\proof  Assume first that $[L:K] \neq 21$ or $30$. 
In all remaining cases, $[L:K]=p^aq^b$ for some distinct primes $p$ and $q$, and $a, b \geq 1$. Since $L/K$ is abelian,
we can find two abelian subextensions of $L$, 
$L_1/K$ and $L_2/K$  of degrees $p^a$ and $q^b$ respectively, such that $L=L_1L_2$.  

We have seen in \ref{pro.boundgeneral} that when ${\rm char}(K) \neq 2$ and $p^a+q^b +1 \leq 10$, there exist infinitely many elliptic curves $E/K$
with pairwise distinct $j$-invariants and with a new point $P$ over $L_1$ and a new point $Q$ over $L_2$. The same result
is proved in Theorem \ref{pro.boundgeneral2} when ${\rm char}(K) = 2$ when $[L:K]=15 $ (use $3+4=7$ and $3+5=8$), and in 
Theorem \ref{thm.d=9} in the other two cases (use $2+7=4+5=9$).  Then 
the statement of the proposition follows immediately from Proposition \ref{pro.composition} below, in the case $e=1$.

Assume now that ${\rm char}(K) \neq 2$  and that $[L:K]= 21$ or $30$. In this case, since $3+7=2+3+5=10$, 
\ref{pro.boundgeneral} shows that there exist infinitely many curves $E/K$ of genus $1$
with pairwise distinct $j$-invariants and with new points of degrees $3$ and $7$, and degrees $2$, $3$, and $5$, respectively.
Since a genus $1$ curve $E/K$ with points of coprime degrees has a $K$-rational point, we find that the curves of genus $1$ are in fact elliptic curves.
Then the statement of the proposition follows immediately from \ref{pro.composition} below, in the case $e=1$.
\qed

\smallskip
Our next proposition shows that 
two new points on an elliptic curve can often be used to construct a third new point in a larger extension.

\begin{proposition} \label{pro.composition}
Let $K$ be any field, and let $\overline{K}$ denote an algebraic closure of $K$.  
Let $L_1/K$ and $L_2/K$ be two extensions contained
in $\overline{K}$ such that $L_1 \not\subseteq L_2$ and
$L_2\not\subseteq L_1$. 
Let $L:=L_1 L_2$. Let $E/K$ be an elliptic curve (or any abelian
variety) with new points $P\in
E(L_1)$ over $L_1$ and $ Q \in E(L_2)$ over $L_2$. Denote by $K(P+Q)$
the subfield of $L$ generated by the coordinates of the point $P+Q \in
E(L)$. Then $K(P+Q)$ cannot be contained in either $L_1$ or $L_2$.
If moreover $L_1$ and $ L_2$ are linearly disjoint over $K$, then
$[K(P+Q):K]\ge \max\{ [L_1: K], [L_2: K]\}$. 

Assume now that $L_1/K$ and $L_2/K$ are Galois and linearly disjoint, and let $e(L_i/K)$ denote the exponent of the Galois group of $L_i/K$.
Set $e:=\gcd(e(L_1/K), e(L_2/K))$.  
If  $E(K)$ does not contain any non-trivial torsion point of order dividing $e$ (this always holds if $e=1$),
then the point $P+Q \in E(L)$ is a new point over $L$.

When $E(K)$ contains a non-trivial torsion point of prime order $p$ and $K$ is a number field, then there exists a finite set $S$ of  Galois extensions  
such that if $L_1$ and $ L_2$ are distinct cyclic extensions of degree $p$ and $L_1, L_2 \notin S$, then $P+Q \in E(L)$ is a new point over $L$.
\end{proposition}

\proof 
Let $L_3:=K(P+Q)$. 
The equality $(P+Q)+(-Q)=P$ shows that $L_1 \subseteq L_2L_3$. 
In particular, if $L_3\subseteq L_2$, then we find that $L_1 \subseteq
L_2$, which is a contradiction. Since $L_1 \subseteq L_2L_3$, we find that $L_3L_2=L$.
If $L_1$ and $ L_2$ are linearly disjoint,
then $[L_1:K][L_2:K]=[L:K]\le [L_3:K][L_2:K]$, so $[L_1:K]\le [L_3:K]$. 

When both $L_1/K$ and $L_2/K$ are Galois and linearly disjoint, $L/K$ is a Galois extension with Galois group $G$ which we identify with ${\rm Gal}(L_1/K) \times {\rm Gal}(L_2/K)$.
Assume that $K(P+Q) \neq L$. Then there exists 
$\sigma \in {\rm Gal}(L_1/K) \times \{{\rm id}\}$, with $\sigma \neq ({\rm id},{\rm id})$, and 
$\tau \in  \{{\rm id}\}  \times {\rm Gal}(L_2/K)$, with $\tau \neq ({\rm id},{\rm id})$,  such that 
$\sigma\tau(P+Q)=P+Q$.  
By hypothesis, $\sigma(Q)=Q$, $\tau(P)=P$. It follows that $\sigma(P)+\tau(Q)=P+Q$, so that $\sigma(P)-P= Q -\tau(Q)$. 
Since $\sigma(P)-P \in E(L_1)$ and $Q -\tau(Q) \in E(L_2)$, we find that $T:=\sigma(P)-P= Q -\tau(Q)$ is in fact in $E(K)$.
Since by assumption $P$ is a new point over $L_1$, we have $\sigma(P) \neq P$, so that $T$ is not trivial.

We claim that ${\rm ord}(\sigma)={\rm ord}(\tau)$. Indeed, for any positive power $d$, we still have $\sigma^d\tau^d(P+Q)=P+Q$, 
which implies that $\sigma^d(P)-P= Q -\tau^d(Q)$. If $\sigma^d \neq {\rm id}$, we find that $\sigma^d(P)\neq P$, so that $\tau^d \neq {\rm id}$.
Since the argument is symmetric in $\sigma $ and $\tau$, our claim follows.

Let then $m:={\rm ord}(\sigma)={\rm ord}(\tau)$,  and let $[m]:E \to E$ denote the multiplication-by-$m$ morphism. 
{\it We claim that   $[m]P \in E(L_1^{\langle \sigma \rangle})$ and $ [m]Q\in E(L_2^{\langle \tau \rangle})$, and that $T \in E(K)$ is a torsion point of order $m$.} 
Indeed, $\sigma^i(P)-P \in E(K)$ since $\sigma^i(P)-P =\sigma^{i-1}(T)+\dots+\sigma(T)+T=[i]T$.  
By construction, $\sum_{i=1}^m \sigma^i(P) \in E(L_1^{\langle \sigma \rangle})$. 
Hence, $[m]P=\sum_{i=1}^m \sigma^i(P)- \sum_{i=1}^{m-1} (\sigma^i(P)-P)$ is also in $E(L_1^{\langle \sigma \rangle})$.
Moreover, we have $P=\sigma(P) -T$, so that $[m]P=\sigma([m]P)-[m]T$, and since $\sigma([m]P)=[m]P$, we find that $[m]T$ is trivial.
To see that $T$ has order $m$, recall that $\sigma^d(P)-P= [d]T$, so that if $[d]T$ is trivial, then so is $\sigma^d$.

 Assume now that $L_1$ and $L_2$
are cyclic of order $p$. It follows from the above discussion that if $K(P+Q) \neq L$, then $[p]P \in E(K)$ and $[p]Q \in E(K)$.
When $K$ is a number field, we can use the following fact:
if $m$ is any integer and $M$ denotes the smallest extension that contains the fields of definition of all points $R$ such that $[m]R \in E(K)$, then the extension $M/K$ is finite
(see e.g., \cite{Lang}, Chapter 6, Propositions 1.1 and 2.3).
 \qed 

\begin{remark} It is conjectured in \cite{DFK}, 1.2, that if $p \geq 7 $ is prime and $E/{\mathbb Q}$ is an elliptic curve, then there exist only finitely many cyclic extensions 
$L/{\mathbb Q}$ of degree $p$ such that the rank of $E(L)$ is bigger than the rank of $E({\mathbb Q})$.
\end{remark}

\begin{example} Let $K$ be a field of characteristic different from $2$. 
Consider an elliptic curve $E/K$ given by  $y^2=x(x^2+a_2x+a_4)$.
Let $T:=(0,0)$. Let $F/K$ be a quadratic extension and let 
$\sigma$ be the generator of $\mathrm{Gal}(F/K)$. One can show that 
there exists $P\in E(\overline{K})$ such that $K(P)=F$ and 
\begin{equation} \label{Ps} 
P-\sigma(P)=T 
\end{equation} 
if and only if either $F=K(\sqrt{d})$ for some 
$d\in K$ such that $d(d^2-2a_2d+a_2^2-4a_4)$ is a square in $K$,
or $F=K(\sqrt{a_2^2-4a_4})$, in which case we set  $d:=0$ in the next sentence.
The point $P$ is then 
$$P=(\alpha, \ \pm \sqrt{d}\kern 1pt \alpha)$$ 
where $\alpha$ (possibly in $K$) is a root of $T^2+(a_2-d)T+a_4\in K[T]$. 
If $P$ and $ Q$ are two points satisfying \eqref{Ps} with $K(P)\ne K(Q)$, 
then by the computations in the proof of Proposition~\ref{pro.composition}, 
$K(P+Q)$ is a quadratic extension different from $K(P)$ 
and $K(Q)$. Note that $y^2=x(x^2-2a_2x+a_2^2-4a_4)$ is an explicit equation for the quotient $E/ \! \left< T\right>$.
\end{example}

\begin{example} Let $K$ be a field of characteristic different from $2$. 
Consider the elliptic curve $E/K$ given by  $y^2=x^3+ax^2+bx+c$. 
Assume that $x^3+ax^2+bx+c$ is irreducible in $K[x]$ and has a splitting field of degree $6$. Then the three non-trivial points $P_1,P_2$, and $P_3$ of order $2$ 
are defined over three  different (non-Galois) cubic extensions of $K$, and $P_1+P_2=-P_3$.
\end{example}

\end{section}
\begin{section}{Finiteness of the number of curves with a new point over $L$} \label{sec.CHM}

Fix an integer $g \geq 2$. In  \cite{CHM}, Theorem 1.2,  Caporaso, Harris, and Mazur,  show that
if the Strong Lang Conjecture holds, then 
there exists an integer $N(g)$ such that for any number field
$F$ there are only finitely many smooth projective curves of genus $g$ defined over $F$ with more
than $N(g)$ $F$-rational points. Let us note that the above statement implies the following.

\begin{emp} \label{CHM} {\it Fix an integer $g \geq 2$. If the Strong Lang Conjecture holds, then there exists an integer $c(g)$,
depending on $g$ only, such that given a number field $K$ and any finite collection of finite extensions $L_i/K$ (the extensions $L_i$ need not be distinct) with $\sum_{i=1}^t [L_i:K]  \geq c(g)$, 
then there exist only {\it finitely many} smooth proper geometrically connected curves $X/K$ of genus $g$ with 
distinct closed points $P_1, \dots, P_t$ in $X$ such that the residue field $K(P_i)/K$ is isomorphic to $L_i/K$ for $i=1,\dots, t$. }
\end{emp}

To prove the above statement, we show that in fact it is possible to take $c(g)=N(g)$. Indeed, suppose given a
smooth proper geometrically connected curve $X/K$ of genus $g$ with 
distinct closed points $P_1, \dots, P_t$ in $X$ such that the residue field $K(P_i)/K$ is isomorphic to $L_i/K$ for $i=1,\dots, t$.
Then over the smallest Galois extension $L$ of $K$ containing $L_1,\dots, L_t$, we find that $X(L)$ contains $\sum_{i=1}^t [L_i:K] $ points,
since each closed point $P_i$ corresponds to $[L_i:K]$ distinct points in $X(L)$. 
To conclude, we use the fact that given a curve $Y/K$ of genus $g \geq 2$, the isomorphism classes of curves $Y'/K$ which become $L$-isomorphic to $Y/K$ over a Galois extension $L/K$ are in bijection with the elements of the set $H^1({\rm Gal}(L/K), {\rm Aut}_{\overline{K}}(Y_{\overline{K}}))$, 
and that this set is finite since both ${\rm Gal}(L/K)$ and $ {\rm Aut}_{\overline{K}}(Y_{\overline{K}})$ are finite.

In fact, it turns out that the statement \ref{CHM}  is equivalent to the statement of Theorem 1.2 in \cite{CHM}. Indeed, assuming that
the statement \ref{CHM} holds,  consider a number field $F$ and a smooth curve $X/F$ of genus $g$ defined over $F$ with $t \geq c(g)$
$F$-rational points. Applying \ref{CHM} to the case $K=F$ and $L_i=K$ for all $i=1,\dots, t$, we find that the number of such curves $X/F$ must be finite.

Assuming that an integer $c(g)$ exists as in \ref{CHM}, there also  exists {\it an integer $d(g)$ such that given a number field $K$ and an extension $L/K$ with $[L:K]  \geq d(g)$, 
then there exist only {\it finitely many} smooth proper geometrically connected curves $X/K$ of genus $g$ with a
closed point $P$ in $X$ such that the residue field $K(P)/K$ is isomorphic to $L/K$.} It is clear by definition that 
the minimal such integer $d(g)$ is such that $d(g) \leq c(g)$.
It follows from Corollary \ref{cor.infinitetelymanycurves}  
that $d(g)>6(g+1)$. It is known that $N(g) \geq  16(g+1)$ (see, e.g., \cite{Cap}, Section 5, for a proof of this fact due to A. Brumer).

It is known that no analogue of the integer $N(g)$ can be defined in the function field setting \cite{CUV}. 
We show below using a variant on the example in \cite{CUV} that no analogue of the integer $d(g)$   can be defined either in the function field setting
in positive characteristic. 

\begin{proposition} \label{curves} 
Let $p>2$ be prime. Let $K$ be an infinite field with ${\rm char}(K)=p$.  Let $g \geq 1$ be  coprime to $p$.
Suppose that $d>2$ is an integer such that for some $m \geq 1$, $d$ divides $p^m+1$. Let $L/K$ be a Kummer extension  of degree $d$. 
When either $g$ or $d$ is odd, there exist infinitely many hyperelliptic curves 
of genus $g$  over $K$, pairwise non-isomorphic over $\overline{K}$,  with a new point over $L$ and a $K$-rational point.
\end{proposition}

\begin{proof}
By hypothesis,  we can write $L=K(\theta)$ with $\theta= z^d$ for some 
element $z\in K$. When $d$ is odd, we can find such a $\theta$ with $z= u^2$ for some $u \in K$. 
Let $a:=\theta^{p^m+1}=z^{(p^{m}+1)/d}\in K$  and assume that $a^g\ne 1$. Let  $C_a/K$ denote the hyperelliptic curve of genus $g$ defined by the equation
$$y^2=x(x^g+1)(x^g+ a^g). $$
This curve always has a $K$-rational point, and we now show that it has a special new point over $L$.
Let us define a square root of  $\theta^{g+1}$ in $L$ as follows. 
If $g$ is odd, then we simply take $\theta^{(g+1)/2}$.  
Suppose now that $g$ is even. By hypothesis,  $d$ is then odd.  
Let $\theta_1:=u^{-1}\theta^{(d+1)/2}$. Then 
$\theta=\theta_1^2$ and we define $\theta^{(g+1)/2}$ to be 
$\theta_1^{g+1} \in L$. 
The reader will verify that $C_a$ admits the special new point
$$P:=(x,y) = \left(\theta, \theta^{(g+1)/2} (\theta^g+1)^{(p^{m}+1)/2} \right)$$
over the field $L$. 

Given any $\gamma \in K^*$, let $\theta':=\gamma^2 \theta$, and $z':=\gamma^{2d} z$. By construction, $z'^2=(\gamma u)^2$. Clearly $L= K(\theta')$. 
Let $a':=(z')^{(p^{m}+1)/d}$. Then $a'^g=\gamma^{2g(p^{m}+1)} a^g $. 
Since $K$ is assumed to be infinite, we find that if $a^g=1$, then there exists $\gamma \in K^*$ such that $\gamma^{2g(p^{m}+1)}\neq 1$, and thus such that $a'^g \neq 1$. In fact, all but finitely many $\gamma \in K^*$ are such that  $a'^g \neq 1$, and for any  $\gamma$ with $a'^g \neq 1$, 
the curve $C_{a'}$ has genus $g$, and also has a new special point over $L$.

It remains to prove that there exists an infinite subset $S$ of $K^*$ such that the curves in the set $\{C_{a'}\}_{\gamma \in S}$ are pairwise 
non-isomorphic 
over $\overline{K}$. This is done in Lemma \ref{genus.small}.
\end{proof}

\begin{remark} 
Let $A$ be a noetherian factorial domain of positive dimension, and let $K$ be its field of fractions. 
Then $K$ has non-trivial Kummer extensions of all possible degrees, and Proposition  \ref{curves} can be applied to $K$. 
Indeed, 
let $d>1$ be any integer. 
Let ${\mathfrak m}$ be a maximal ideal of $A$. Since $A$ is noetherian of positive dimension, we can choose 
$\alpha \in {\mathfrak m} \setminus {\mathfrak m}^2$. 
The polynomial $x^d-\alpha \in A[x]$ is then irreducible in $A[x]$ by the Eisenstein criterion. 
Since $A$ is factorial, Gauss' Lemma implies that $x^d-\alpha $ is also irreducible in $K[x]$.
It would be very interesting to know whether  an analogue of Proposition \ref{curves}  when $K$ has characteristic $0$ can be true.
\end{remark}

\begin{emp} \label{evidence} 
We now provide some evidence that the analogue of the integer $d(g)$ does not exist for $g=1$. 
More precisely, let $\sqrt[\ell]{m}$ denote a root of  $x^\ell -m$, with $\ell >2 $ prime and $m$ an integer such that $x^\ell -m$ is irreducible over ${\mathbb Q}$. 
Let $L:={\mathbb Q}(\sqrt[\ell]{m})$. We suggest below that there exist infinitely many elliptic curves $E/{\mathbb Q}$ with a new point over $L$.

For this we will need to assume several conjectures. First, 
assume   that there exist infinitely many primes $p$ of the form $u^2+64$ with $u \in {\mathbb N}$ (see \cite{DW}, Conjecture 1). 

It is known that for each such prime $p$, there exists an elliptic curve $E_p/{\mathbb Q}$ of conductor $p$
with a $2$-torsion point over ${\mathbb Q}$ and split multiplicative reduction at $p$. Such a curve is called a Neumann-Setzer elliptic curve (\cite{Neu}, \cite{Set}). 
We suggest that there exist infinitely many primes $p$ as above such that the elliptic curve $E_p/{\mathbb Q}$ has a new point over $L$.

New points over $L$ will be obtained using the Parity Conjecture, 
which states that over a number field $K$,  the rank ${\rm rank}(E/K)$ of the Mordell--Weil group of an elliptic curve $E/K$ and the root number $\omega(E/K)$
of the elliptic curve are related by the formula $(-1)^{{\rm rank}(E/K)}=\omega(E/K)$ (see, e.g., \cite{Dok}, 1.1). 
The Parity Conjecture implies that if  $E/{\mathbb Q}$ is an elliptic curve such that $\omega(E/{\mathbb Q}) \neq \omega(E_L/L)$, then $E$ must have a new point over $L$.
Indeed, the former condition, under the Parity Conjecture, implies that the rank of $E(L)$ is larger than the rank of $E({\mathbb Q})$. Since ${\mathbb Q}$ is the 
only proper subfield of $L$ because $[L:{\mathbb Q}]$ is prime by hypothesis, we conclude that $E$ must have a new point over $L$, as desired. 

The root number of an elliptic curve is obtained as a product of local root numbers. The local root number is equal to $-1$ when the place is archimedean, 
or when it is non-archimedean with split multiplicative reduction (see, e.g., \cite{Dok}, 3.1). The local root number is $+1$ when the elliptic curve has good reduction at the place.

Let $E_p/{\mathbb Q}$ be an elliptic curve of prime conductor $p$, with split multiplicative reduction at $p$. 
It follows that $\omega(E/{\mathbb Q}) = 1$.  When the curve has a ${\mathbb Q}$-rational  $2$-torsion point, the rank of $E_p({\mathbb Q})$ is $0$ (see \cite{Maz}, 9.10). Let us now compute the root number $\omega(E_L/L)$.

The extension $L/{\mathbb Q}$ has $1+\frac{\ell-1}{2}$ archimedean places. 
Let $s$ denote the number of  prime ideals of ${\mathcal O}_L$ that contain $p$.
Since the elliptic curve $E_L/L$ has split multiplicative reduction at all places above $(p)$, we find that 
$$ \omega(E_L/L)= (-1)^{(\ell+1)/2} (-1)^s.$$
In our discussion, $m$ is fixed and $p$ varies. Thus except for finitely many exceptions, we can assume that $p$ is coprime to $m$. 
For such a prime $p$, ${\mathbb Z}[\sqrt[\ell]{m}]$ localized at a prime above $p$ is a discrete valuation ring and, thus, 
the factorization of $(p)$ in ${\mathcal O}_L$ is obtained using the factorization 
of  $x^\ell -m$ modulo $p$.

In our discussion, $\ell$ is fixed. Thus it seems reasonable to conjecture that there exist infinitely many primes $p$ of the form $u^2+64$
such that  $\ell$ is coprime to $p-1$ (in fact, if   $\ell\equiv 3,5,$ or $6$, modulo $7$, then $\ell \nmid p-1$). Then $m\equiv M^\ell \pmod{p}$ for some integer $M$, and the factorization of $x^\ell-m$ modulo $p$
 is obtained from the factorization of $x^\ell -1$ in ${\mathbb F}_p[x]$. 
For this we compute the order $f$ of $p$ modulo $\ell$, and find that $x^\ell-1$ has $s:=1+\frac{\ell-1}{f}$ irreducible factors, so that 
there are exactly $1+\frac{\ell-1}{f}$ prime ideals ${\mathcal O}_L$ above $(p)$.
It seems reasonable to conjecture that among the infinitely many primes $p$ of the form $u^2+64$
such that  $\ell$ is coprime to $p-1$, infinitely many of them are such that $s$ is even, and infinitely many are such that $s$ is odd. 

When considering the set of all primes $p \neq \ell$ and not only the much smaller subset of primes $p$ of the form $u^2+64$, Chebotarev's Theorem
can be applied to the Galois extension ${\mathbb Q}(\zeta_\ell)$ to obtain the following result: {\it 
Assume that $\ell\equiv 3,5,$ or $6$, modulo $7$, and that $\ell \equiv 3 \pmod{4}$. Then
the density of primes $p$ such that $s$ is even is $1/2$, and the density of primes $p$ such that $s$ is odd is also $1/2$.}
Indeed, the Galois group $G$ of ${\mathbb Q}(\zeta_\ell)/{\mathbb Q}$ is cyclic of order $\ell-1$, and 
the number of prime ideals above $(p)$ in ${\mathbb Z}[\zeta_\ell]$ is $(\ell-1)/{\rm ord}(\sigma)$, where $\sigma $ is the Frobenius element associated with $(p)$. When $\ell \equiv 3 \pmod{4}$, 
$(\ell-1)/{\rm ord}(\sigma)$ is even if and only if $\sigma $
belongs to the unique subgroup $H$ of index $2$ in $G$. It follows from Chebotarev's Theorem that the density of primes $p$
such that the number of prime ideals above $(p)$ in ${\mathbb Z}[\zeta_\ell]$ is even is $|H|/|G|=1/2$.

If  the above conjectures are true, then there exists a sequence of Kummer extensions $L$ of increasing degree $\ell$, and for each such field, 
there exist infinitely many primes $p$ as above such that the elliptic curve $E_p/{\mathbb Q}$ has a new point over $L$.

We show in \ref{case ell=11}, using a completely different method, that when 
$\ell=11$ there exist infinitely many pairwise non-isomorphic elliptic curves over ${\mathbb Q}$ with a new point over $L$.
(This method might also apply to fields $L:=K(\sqrt[\ell]{m})$ of characteristic prime to $\ell=11$.) 
That the same statement holds when $\ell=5$ or $7$ follows from \ref{cor.elliptic}.
\end{emp}

\begin{remark} We discuss in this remark some analogies between statements obtained by reversing the roles played by a curve of genus $g$ and an extension of degree $d$. Let $L/K$ be a fixed extension of degree $d>1$. 
Consider the set ${\mathcal G}$ of all integers $g \geq 1$ such that there exists a smooth proper geometrically connected curve $X/K$ of genus $g$
with a new point in $L$. When $K$ is infinite of characteristic different from $2$, Theorem \ref{thm.bound} shows that if $g \geq d/4$ (and $g\geq 2$), then $g \in {\mathcal G}$. Moreover, for each such `large' $g$, there exist infinitely many pairwise non-isomorphic curves $X/K$ of genus $g$ with a new point on $L$.  

An analogous result holds true when the roles of $L/K$ and $X/K$ are reversed. Indeed, fix a curve $X/K$ of genus $g>1$, and consider the set 
${\mathcal D}$ of all integers $d \geq 1$ such that there exists a finite extension $L/K$ of degree $d$ such that $X/K$ has a new point over $L$. 
When $K$ is a number field, we noted in \cite{GLL1}, 7.5, that if $X/K$ has index $1$, then there exists a bound $d_0=d_0(g)$ such that if $d \geq d_0$, then $d \in {\mathcal D}$.
This result uses Hilbert's Irreducibility Theorem on a $K$-morphism $X \to {\mathbb P}^1$ of degree $d$. We note that for all such `large' $d$, 
Hilbert's Irreducibility Theorem implies that there exists infinitely many fibers $\Spec L$ which are irreducible, and since the curve has genus $g>1$, 
we can use Mordell's Conjecture to conclude that there are infinitely many pairwise non-isomorphic such fibers.

Recall now the Caporaso-Harris-Mazur result \ref{CHM} (under the Strong Lang Conjecture), which states that if $K$ is a number field and $g>1$ is fixed, then there exists an integer $d=d(g)$
such that if $L/K$ is an extension of degree at least $d$, then there exist only finitely many smooth proper geometrically connected curves $X/K$  of genus $g$ with a new point over $L$. A possible analogue of this statement when the roles of $L/K$ and $X/K$ are reversed could 
be the following theorem of Frey \cite{Frey}, Proposition 2, which uses a deep result of Faltings \cite{Fal} (see also \cite{A-V}, Corollary 3.4). 
Recall that the {\it gonality} of a curve $X/K$
is the smallest integer $\gamma$ such that there exists a non-constant $K$-morphism 
$X \to {\mathbb P}^1_{K}$ 
of degree $\gamma$.

\begin{emp} \label{gonality} {\it Let $K$ be a number field and $d>1$ be fixed. If $X/K$ is a smooth proper geometrically connected curve
of {\it gonality} at least $2d+1$, then there exist only finitely many extensions $L/K$ of degree $d$ over which $X/K$ has a new point.}
\end{emp} 
\end{remark}

\end{section}
\begin{section}{Examples} \label{ex.cyc}
 Let $K={\mathbb Q}$. Even for standard families of number fields $L/{\mathbb Q}$, such as the cyclotomic fields ${\mathbb Q}(\zeta_n)$ 
or their maximal totally real subfields ${\mathbb Q}(\zeta_n)^+$, it is not known in general whether there exists an elliptic curve $E/{\mathbb Q}$
with a new point over $L$. 

Some sporadic examples are known, such as when $n=43$, $67$, and $163$, where the existence of an elliptic curve $E/{\mathbb Q}$ with complex multiplication and a new torsion point of order $n$
over ${\mathbb Q}(\zeta_n)^+$ is established in \cite{DGJ}, Lemma 4. 

\begin{emp} 
Since ${\mathbb Q}(\zeta_n)/{\mathbb Q}$ is abelian of degree $\varphi(n)$, we can use Proposition \ref{cor.extension15-16}
when $\varphi(n)=12$ or $20$ to show the existence of infinitely many elliptic curves $E/{\mathbb Q}$ with a new point over ${\mathbb Q}(\zeta_n)$.
Recall that $\varphi(n)=16$ when $n=17$, $32$, $34$, $40$, $48$, or $60$. We do not know of an explicit elliptic curve $E/{\mathbb Q}$ with a new point over ${\mathbb Q}(\zeta_{17})$ (see however \ref{rem.ellipticBSD}). 

When viewing ${\mathbb Q}(\zeta_{32})$, ${\mathbb Q}(\zeta_{40})$, and
${\mathbb Q}(\zeta_{48})$,  as Kummer extensions of degree $4$ over an appropriate subfield of degree $4$,
Proposition \ref{pro.bound} 
shows  that  there exist 
genus $1$ curves with a new point over such ${\mathbb Q}(\zeta_n)$.  
For these three cases where the degree is $16$, we leave it to the reader to exhibit such a genus $1$ curve with a ${\mathbb Q}$-rational point. 

In the case of ${\mathbb Q}(\zeta_{60})$, we proceed using the method of \ref{cor.extension15-16}  and \ref{pro.composition}. 
Let $f(x):=(x^2+1)(x^2+x+1)(x^4+x^3+x^2+x+1)$, which we write as $f(x)=h(x)^2 + \ell(x)$ with 
$\ell(x) =(1/4)(x^3 + 3x^2/2 + x + 15/16)$. It turns out that $\ell(x)$ is irreducible in ${\mathbb Q}[x]$ so that the elliptic curve defined by $
y^2=-\ell(x)$ has no non-trivial ${\mathbb Q}$-rational $2$-torsion points. Thus, the points $P:=(i, h(i))$, $Q:=(\zeta_3,h(\zeta_3))$, and $R:=(\zeta_5, h(\zeta_5))$ add to a point $P+Q+R$ which is a new point over the field ${\mathbb Q}(i,\zeta_3, \zeta_5)={\mathbb Q}(\zeta_{60})$.
\end{emp}

\begin{remark} \label{rem.ellipticBSD}
Given a number field $L$ and an elliptic curve $E/K$, it is sometimes possible to {\it conjecturally} determine that $E/K$ has a new point over $L$
as follows. For every proper maximal subfield $L_0$ of $L/K$, compute the analytic rank of $E$ over $L_0$. If, for each such maximal proper subfield, 
the analytic rank of $E$ over $L$ is strictly larger than the analytic rank over $L_0$, then the Birch and Swinnerton-Dyer Conjecture
would imply that the algebraic rank of $E$ over $L$ is larger than the algebraic rank of $E$ over $L_0$ and, thus, that $E/K$ has a new point over $L$.
(We use here that the ${\mathbb Q}$-vector space $E(L) \otimes {\mathbb Q}$ cannot be the union of finitely many proper subspaces.)

In the table below, for a given field $L/{\mathbb Q}$, we list the Cremona labels of the first elliptic curves $E/{\mathbb Q}$ with a  (conjectural) new point over $L$ found by this method: 
we used Magma \cite{Magma} to test the analytic ranks over relevant subfields of each curve in Cremona's table \cite{Cre} up to a certain conductor. 

\renewcommand{\arraystretch}{1.3}
$$\begin{tabular}{|l|l||l|l|}
\hline
$L$ & $E/{\mathbb Q}$ &$L$ & $E/{\mathbb Q}$\\
\hline 
${\mathbb Q}(\zeta_{17})$ &   139a1, 143a1, 161a1, 168a1 & ${\mathbb Q}(\zeta_{25})$ &   49a1, 82a1, 201b1 \\
\hline 
${\mathbb Q}(\zeta_{19})$ &  17a1, 33a1, 80a1, 124b1 & ${\mathbb Q}(\zeta_{27})$ &   37a1, 73a1, 91a1 \\
\hline 
${\mathbb Q}(\zeta_{23})$ &  89a1, 94a1, 170e1  & ${\mathbb Q}(\zeta_{29})$ &  17a1, 84a1, 264b1\\
\hline
${\mathbb Q}(\zeta_{23})^+$ & 89a1, 197a1, 794b1, 954h1 & ${\mathbb Q}(\zeta_{31})$ & 50a1, 90a1, 136a1\\
\hline
${\mathbb Q}(\zeta_{47})^+$ &  204b1, 786m1 &   ${\mathbb Q}(\zeta_{33})$ &  46a1, 65a1, 69a1  \\
\hline
\end{tabular}
$$

\smallskip
\noindent We note that when the analytic rank of a given elliptic curve $E/K$ is larger than $3$, 
AnalyticRank(E) in Magma only returns an integer that is `probably' the analytic rank of $E/K$.
\end{remark}

\begin{emp}
When $[{\mathbb Q}(\zeta_n):{\mathbb Q}]=12$, it follows from \ref{pro.boundgeneral} that there exist infinitely many curves $X/{\mathbb Q}$ of genus $g=2$ 
with a new point over ${\mathbb Q}(\zeta_n)$. When $[{\mathbb Q}(\zeta_n):{\mathbb Q}]=16$, the same result holds using \ref{pro.bound}
and the fact that ${\mathbb Q}(\zeta_n)$ is a Kummer extension of degree $2$ over  ${\mathbb Q}(\zeta_n)^+$.
We give below the smallest known genus $g \geq 2$ where there exists a curve $X/{\mathbb Q}$ of genus $g$ with a new point over ${\mathbb Q}(\zeta_n)$ for some extensions with  $[{\mathbb Q}(\zeta_n):{\mathbb Q}]\geq 18$.
\renewcommand{\arraystretch}{1.3}
$$\begin{tabular}{|l|l|}
\hline
$g$ & $L$\\
\hline 
\hline 
$2$ & 
${\mathbb Q}(\zeta_{19})$,  ${\mathbb Q}(\zeta_{25})$, ${\mathbb Q}(\zeta_{27})$, ${\mathbb Q}(\zeta_{72})$   \\ 
\hline 
$3$ &  ${\mathbb Q}(\zeta_{33})$, ${\mathbb Q}(\zeta_{35})$, ${\mathbb Q}(\zeta_{44})$,   ${\mathbb Q}(\zeta_{56})$, ${\mathbb Q}(\zeta_{64})$,
 ${\mathbb Q}(\zeta_{84})$,  ${\mathbb Q}(\zeta_{96})$ \\
\hline
$4$ & ${\mathbb Q}(\zeta_{23})$, ${\mathbb Q}(\zeta_{39})$, ${\mathbb Q}(\zeta_{45})$, ${\mathbb Q}(\zeta_{52})$ \\ 
\hline
$5$ & ${\mathbb Q}(\zeta_{29})$, ${\mathbb Q}(\zeta_{31})$ \\
\hline
\end{tabular}
$$

\medskip \noindent
The cases of ${\mathbb Q}(\zeta_{19})$ and ${\mathbb Q}(\zeta_{31})$ are treated in \ref{ex.special}. 
The other cases follow from  
\ref{pro.bound}, including the case of ${\mathbb Q}(\zeta_{27})$,
which we view as a Kummer extension of degree $3$ 
on top of an extension of degree $6$. 
We view ${\mathbb Q}(\zeta_{72})$
as a Kummer extension of degree $6$ over a field of degree $4$. 
We view ${\mathbb Q}(\zeta_{25})$ and ${\mathbb Q}(\zeta_{35})$ as Kummer extensions of degree $5$ and $7$ over a field of degree $4$.
We view ${\mathbb Q}(\zeta_{56})$
and ${\mathbb Q}(\zeta_{84})$ as Kummer extensions of degree $7$ over a field of degree $4$.
We view ${\mathbb Q}(\zeta_{64})$
and ${\mathbb Q}(\zeta_{96})$ as Kummer extensions of degree $8$ over a field of degree $4$.
\end{emp}

\begin{example} \label{ex.special} Let $p$ be a prime, and consider the cyclotomic field ${\mathbb Q}(\zeta_p)$.
Let $f(x)$ denote the minimal polynomial of $\alpha:=\zeta_p(1-\zeta_p)$ over ${\mathbb Q}$. Then (computational evidence indicates that) there exists  $\dell(x) \in {\mathbb Z}[x]$
of degree $(p-3)/2$ such that $f(x)= x^{p-1} + p\dell(x)$, and that moreover, 
  ${\rm ord}_{x-1}(\dell(x))=1$, unless $6$ divides $p-1$, in which case ${\rm ord}_{x-1}(\dell(x))= 2$.
  
  It follows that when $6$ divides $p-1$, the curve $y^2=- p\dell(x)/(x-1)^2$ has the new point $(\alpha, \alpha^{(p-1)/2}/(\alpha-1))$, and when $(p-1)/6$ is odd, it has genus $g$
one less than the bound obtained in  \ref{pro.bound} or  \ref{pro.boundgeneral}. For instance, when $p=19$ and $31$, we find that  $g=2$ and $5$, respectively.
\end{example}
\begin{example} \label{ex.Kummer} 
{\bf (Kummer Extensions.)} 
Consider the field $L:={\mathbb Q}(\sqrt[\ell]{m})$, where $\sqrt[\ell]{m}$ denotes a root of $x^\ell -m \in {\mathbb Z}[x]$.
Assume that $\ell>3$ is prime and that $x^\ell -m$ is irreducible over ${\mathbb Q}$, so that $L/{\mathbb Q}$ does not contain any subextension. Thus, Proposition \ref{pro.bound} does not apply, and 
Theorem \ref{pro.boundgeneral} shows that there exists a hyperelliptic curve of genus $g=\lfloor (\ell-3)/4\rfloor$ with a new point over $L$. 

When $\ell \equiv 3 \mod{4}$,
the method discussed in \ref{rem.weakermethod} allows us to exhibit an explicit such curve, given by the equation  $y^2=4x^{(\ell+1)/2} +mx+4$, 
of genus $(\ell-3)/4$ when $4x^{(\ell+1)/2} +mx+4$ is separable, with the $K$-rational point $(0,2)$
and the new point $(\sqrt[\ell]{m}, (\sqrt[\ell]{m})^{(\ell+1)/2} +2)$.

When $\ell \equiv 1 \mod{4}$, one can exhibit an explicit curve of genus $\lfloor (\ell-3)/4 \rfloor$ with a new point over $L$ as follows.
The minimal polynomial $f(x)$ of 
$$\alpha:=\sqrt[\ell]{m}(1-\sqrt[\ell]{m})$$ is 
of the form
 $f(x)= x^\ell+m\dell(x)$ with 
$\dell(x) \in {\mathbb Z}[x]$ of degree $(\ell-1)/2$.
Indeed, the polynomial $(x^\ell-m)((1-x)^\ell-m)$ is invariant under the map $x \mapsto (1-x)$, and can thus be expressed as a polynomial in the variable $z:=x(1-x)$.  Writing 
$$(x^\ell-m)((1-x)^\ell-m)= (x(1-x))^\ell -mx^\ell -m(1-x)^\ell +m^2 = f(z), $$
we find that $f(z)=z^\ell -mz^{(\ell-1)/2}+\dots +m^2-m$, as desired.
We thus have obtained the curve $y^2=-mx\dell(x)$ with the new point $(\alpha, \alpha^{(\ell+1)/2})$ over $L$ and a ${\mathbb Q}$-rational point. This curve has genus 
$(\ell-3)/4$ when $\ell \equiv 3 \mod{4}$. When $\ell \equiv 1 \mod{4}$, it is very easy to explicitly express $xf(x) = h(x)^2 -\ell(x)$ as in Lemma \ref{lem.decomp} to obtain an explicit curve of genus $\lfloor (\ell-3)/4 \rfloor$ with a new point over $L$.

\begin{emp} \label{case ell=11} As above, $L:={\mathbb Q}(\sqrt[\ell]{m})$, and assume now that $\ell=11$. Then we can slightly improve Theorem \ref{pro.boundgeneral} as follows. Indeed, the minimal polynomial of $\alpha$ is 
$$f(x)=x^{11} + 11m(x^5 - 5x^4 + 7x^3 - 4x^2 + x) + m(m-1),$$ and  
we find that $(\alpha, \alpha^4)$ is a new point on the genus $1$ curve with the ${\mathbb Q}$-rational point $(0,0)$ given by the cubic 
$$y^3 + 11m(yx^2 - 5xy + 7y - 4x^3 + x^2) + m(m-1)x.$$
The $j$-invariant of the Jacobian of this cubic is a non-constant rational function $j(m)$ in the variable $m$. 
Thus, it is possible to find infinitely many integers of the form $a^{11}m_0$, with $a$ a non-zero integer, such that the $j$-invariants $j(a^{11}m_0)$
are pairwise distinct, producing in this way infinitely many pairwise non-isomorphic elliptic curves $E_{a^{11}m_0}/{\mathbb Q}$ with a new point over $L={\mathbb Q}(\sqrt[\ell]{m_0}) = {\mathbb Q}(\sqrt[\ell]{a^\ell m_0})$.
\end{emp}

\begin{emp} When $\ell \geq 13$, we provide in \ref{evidence}  some evidence that there might exist
an elliptic curve $X/{\mathbb Q}$ with a  new point over $L:={\mathbb Q}(\sqrt[\ell]{m})$. We do not know how to exhibit explicit examples of such points. 

When $6$ divides $\ell-1$ and $L:={\mathbb Q}(\sqrt[\ell]{3^*})$ with $3^*:=3(-1)^{(\ell-1)/2}$, the polynomial $\dell(x) \in {\mathbb Q}[x]$ associated to $\alpha$ with $m= 
(-1/3)^{(\ell-1)/2}$ is always divisible by $(x-1/3)^2$ (a similar phenomenon is mentioned in \ref{ex.special}). 
Thus, the curve $y^2=xt(x)/(x-1/3)^2$ has a new point over $L$ and when $(\ell-1)/6$ is odd, it has genus one less than the genus of the curves 
obtained with  \ref{pro.boundgeneral}. 

Let us now exploit the existence of special Fermat quotients when $\ell \equiv 1 \mod 3$ to show that in fact, for some values of $m$ such as $m=2$, there exists
a smooth projective curve $X/{\mathbb Q}$ of genus $(\ell-1)/6$ with a new point over $L$.
Recall that the Fermat curve, given by the equation $x^\ell + y^\ell=z^\ell$, has $\ell-2$ quotients $C_a$ of genus $(\ell-1)/2$, 
given by the equation $y^\ell=(x-1)x^a$ for $a \in [1,\ell-2]$. 

The curve $C_a/{\mathbb Q}$ has the new points $(2, \sqrt[\ell]{2^a})$, $(-1, (-1)^{a+1}\sqrt[\ell]{2})$,
and $(1/2, -1/\sqrt[\ell]{2^{a+1}})$ over $L={\mathbb Q}(\sqrt[\ell]{2})$.
It turns out that when $a^2+a+1 \equiv 0 \mod \ell$, the curve $C_a/K$ has an extra $K$-automorphism $\sigma$ of order $3$ and the three points above form an orbit of $\sigma$. 
The quotient $X:=C_a/\langle \sigma\rangle$ has genus $(\ell-1)/6$ (\cite{H-R}, 4.1). Since the quotient $C_a \to X$ has degree $3$, and $[L:{\mathbb Q}]= \ell \neq 3$, 
we find that $X$ also has a new point over $L$; it also has a ${\mathbb Q}$-rational point since $C_a$ does.
\end{emp}
\end{example}
\begin{remark} 
It is possible to give a slightly  more precise description of the minimal polynomial $f(x)$ of $\alpha$ introduced in \ref{ex.Kummer}. When $\ell \geq 13$, we find that 
\begin{equation*}
\begin{array}{ll}
&f(x)=  x^\ell+ m(m-1)+\ell m x(x-1) \cdot \\
 &  \quad \cdot \left((-1)^{(\ell-3)/2}x^{(\ell-5)/2}+ \cdots + \frac{(\ell-5)(\ell-7)(\ell-10)}{24} x^3 - \frac{(\ell-5) (\ell-7)}{6} x^2 +\frac{(\ell-5)}{2}x -1\right).
\end{array}
\end{equation*}

Consider now the above polynomial $f(x)$ as a polynomial $f(x,m) \in {\mathbb Z}[x,m]$, defining a plane curve. It turns out that this plane curve has always genus $0$, because when we view $f(x,m)=0$ as defining a hyperelliptic curve with equation of the form $m^2 +mg(x)+x^{\ell}=0$ and we complete the square, 
we find that $x^{\ell}-g(x)^2/4= (x-1/4)h(x)^2$ for some polynomial $h(x) \in {\mathbb Z}[x]$, with 
$$h(x) = x^{(\ell-1)/2} - \frac{\ell^2-1}{8}x^{(\ell-3)/2}+ \dots +(-1)^{(\ell-1)/2}\frac{(\ell-3)(\ell-4)}{2} x^2 -(\ell-2)x+1.$$ The splitting fields of the polynomials $g(x)$ and $h(x)$ are both isomorphic to the totally real cyclotomic subfield ${\mathbb Q}(\zeta_\ell)^+$. It turns out that $h(x)$ is the minimal polynomial of $1/w$, 
where $w:= \zeta_\ell + \zeta_{\ell}^{-1}+2$.
\end{remark}

\begin{example}  {\bf (Artin--Schreier Extensions).} Let $K$ be a field of characteristic different from $2$ or $3$. 
Consider an Artin--Schreier extension $L=K(\alpha)$, where $\alpha $ is a root of an irreducible polynomial $x^\ell -(ax+b) \in K[x]$ with $ab\neq 0$, 
and in this example $\ell  \in {\mathbb N}$ need not be prime.
When $\ell+1 $ is divisible by $3$ (resp.~by $4$), we find that the elliptic curve
 $y^3=ax^2+bx$ has the  new point $(\alpha, \alpha^{(\ell+1)/3})$, and that $y^4=ax^2+bx$ has the new point $(\alpha, \alpha^{(\ell+1)/4})$. 
Note that these two elliptic curves are not isomorphic over $\overline{K}$ and thus when $\ell+1$ is divisible by $12$ we have two different elliptic curves with a new point over $L$. When $\ell-2$ is divisible by $4$, we find that the elliptic curve $y^4=ax+bx^2$ has the new point $(1/\alpha, \alpha^{(\ell-2)/4})$
over $L$. 

When $\ell=12$, the above does not apply. In this case however,   $\alpha^5$ is a root of
$f(x)=x^{12}-5ab^2x^5-5a^3bx^3-a^5x-b^5$ (a machine easily  checks that $f(x^5) $ is divisible by $x^{12}-ax-b$ in $K[a,b][x]$).
We then find that the cubic curve $y^3 -5b^2yx = 5bx^3 +a^4x + b^5$ has the new point $(a\alpha^5, \alpha^{20})$ 
over $K(\alpha^5)$.  The discriminant of the Jacobian of this genus $1$ curve over ${\mathbb Z}[a,b]$ is $\Delta=-25 b^2  (432 a^{24} + 93535 a^{12} b^{11} + 200 b^{22})$.

When   $\ell=13$,   $\alpha^5$ is a root of 
$f(x)=x^{13}-5abx^8+5a^2b^2x^3-a^5x-b^5$.
We then find that the curve $y^2 -5abyx^2 = -5a^2b^2x^4 +a^5x^2 + b^5x$ has the new point $(\alpha^5, \alpha^{35})$ 
over $K(\alpha^5)$ and the $K$-rational point $(0,0)$.  The discriminant of the Jacobian of this genus $1$ curve over ${\mathbb Z}[a,b]$ is $\Delta= -5a^{4}b^{12}(16a^{13}+135b^{12})$.
 
When $\ell=6g+4$, $\alpha^3$ is the root of 
$f(x)= x^{6g+4} - 3ax^{4g+3} + 3a^2x^{2g+2} - a^3x - b^3$, and we find  a curve (of genus $g$ in general) of the form $y^2-3ayx^{g+1}=-3a^2x^{2g+2}+a^3x+b^3$ with a new point over $L$. When $\ell=6g+5$, $\alpha^3$ is the root of 
$f(x)= x^{6g+5} - 3abx^{2g+2} - a^3x - b^3$.
\end{example}

\end{section}

\begin{section}{Non-constant families} 

Let $K$ be an infinite field.
We recall here a method for showing that a one-parameter family of smooth projective curves is not constant. The lemmas in this section are used
in \ref{orbits}, \ref{pro.bound}, and 
\ref{curves}.   

Let $U$ be an affine integral smooth curve over $K$. Let ${\mathcal X} \to U$ be a family of smooth projective curves of genus $g\geq 1$, that is, ${\mathcal X} \to U$ is a smooth projective morphism whose fibers  are geometrically connected curves of genus $g$. Let $S$ be a regular compactification of $U$. To prove that the families we are dealing with have infinitely many non isomorphic members, we will show that they have a limit fiber above some point of $s\in S\setminus U$ which is a stable non-smooth curve. Let us explain the principle. 

Let ${\mathcal M_{g,K}}/K$  be the proper Deligne--Mumford stack 
of stable curves of genus
$g \geq 2 $ over $\Spec K$ (see \cite{D-M}, 5.1). This stack  admits a coarse moduli space $ M_{g,K}/K$ which is proper over $K$.
Denote by $ M^0_{g,K}$ the open subset  corresponding to smooth stable curves  of genus $g$.
By  definition, there exists a $K$-morphism $U \to  M^0_{g,K}$
which sends a geometric point $u $ of $ U$ to the point of $M^0_{g,K}$ which represents the isomorphism class of the fiber of ${\mathcal X} \to U$ above $u$.

\begin{emp} 
\label{semistable}
{\it Suppose that  there are infinitely many closed fibers of ${\mathcal X} \to U$ which are isomorphic over $\overline{K}$}. Then the morphism
$U \to  M^0_{g,K}$ is not quasi-finite. As $\dim U=1$, the map is constant. Since $ M_{g,K}/K$ is proper, the morphism $U \to  M^0_{g,K}$ extends to a morphism 
$S \to M_{g,K}$, which is also constant. 
{\it We claim that the generic fiber of ${\mathcal X} \to U$
has then potentially good reduction over $\Spec \cO_{S,s}$} for any $s\in S\setminus U$. Indeed, by the stable reduction theorem, there exists an integral proper regular curve $T/K$, a  finite surjective morphism $\rho: T\to S$ and a stable curve $\mathcal X'\to T$ of genus $g$ such that $\mathcal X\times_U \rho^{-1}(U) = \mathcal X'\times_T \rho^{-1}(U)$. The morphism $T\to M_{g, K}$ induced by $\mathcal X'\to T$ is the same as the composition $T\to S\to M_{g,K}$ because $M_{g,K}$ is separated.  Therefore $T\to M_{g,K}$ has constant image in $M_{g,K}^0$ and $\mathcal X'\to T$ is smooth. 

\medskip

We are going to use \ref{semistable} in the proof of the following two lemmas.
\end{emp} 
 \begin{lemma} \label{genus.small}
Let $g \geq 1$. Let $K$ be an infinite field of characteristic different from $2$ and prime to $g$.  
\begin{enumerate}[\rm (1)]
\item For any $\lambda \in K^*$ with $\lambda^g \neq 1$,
let $X_\lambda$ denote the hyperelliptic curve over $K$ defined by the equation 
$y^2=x(x^g+1)(x^g+\lambda^g)$.
Then there exists an infinite 
subset $S$ of $K^*$ such that the curves $X_\lambda$, $\lambda\in S$, are pairwise non-isomorphic over $\overline{K}$.
\item The hyperelliptic curve $X/K(t)$ defined by the equation $y^2=x(x^g+1)(x^g+t^g)$ does not have potentially good reduction
over $\Spec K[t]_{(t)}$.
\end{enumerate}
\end{lemma}

\begin{lemma} \label{lem.semistable} Let $K$ be an infinite field of
characteristic different from $2$. Let $N-2\ge k\ge 2$ be two integers, 
with $N-k$ and $k$ both prime to the characteristic exponent of $K$. 
\begin{enumerate}[\rm (1)] 
\item Let $X_\lambda/K$ denote the  hyperelliptic curve defined by 
$ y^2=x^N+ \lambda x^k+1$
(smooth except for finitely values of $\lambda \in \overline{K}^*$). 
Then there exists an infinite 
subset $S$ of $K^*$ such that the curves $X_\lambda$, $\lambda\in S$, are pairwise non-isomorphic over $\overline{K}$.
\item The hyperelliptic curve $X/K(t)$ defined by the equation $ y^2=tx^N+  x^k+t $ does not have potentially good reduction
over $\Spec K[t]_{(t)}$. 
\end{enumerate}
\end{lemma}

The statement \ref{semistable} is used in the proof of both lemmas to show that (2) implies (1) when $g \geq 2$. 
To prove (2), we compute enough of a regular model for the curve $X/K(t)$, over an extension $K(s)$ of the form $s^a=t$ for some $a\geq 1$,
to show that the semi-stable reduction of the curve $X/K(t)$ cannot be  good. The stable reduction of a hyperelliptic curve
over $K(t)$ with ${\rm char}(K) \neq 2$ is well understood (see, e.g., \cite{Bosch}). Here we follow a 
general method described in \cite{LL}, 3.9 and 3.10. 
 
Let $y^2=f(x)$ denote an equation for $X$, with $f(x)$ separable. Make a finite separable extension $F/K(t)$ such that the polynomial $f(x)$ has all its roots in $F$. 
Consider the quotient $\pi: X \to {\mathbb P}_{F}^1$ of $X$ by the hyperelliptic involution, and construct a model ${\mathcal Y}/{\cO_F}$
of ${\mathbb P}_{F}^1$
such that the points in the branch locus of $\pi$ specialize to distinct points in the smooth locus of the special fiber of ${\mathcal Y}$. Let ${\mathcal X}/\cO_F$ 
denote the integral closure of ${\mathcal Y}$ in the function field of $X_F/F$. Consider the natural finite morphism $\pi: {\mathcal X} \to {\mathcal Y}$. 
The curve $X/K$ cannot have potentially good reduction if 
we can find a model ${\mathcal Y}$ with two components $\Gamma$ and $\Gamma'$ in the special fiber of ${\mathcal Y}$ such that the preimages $\pi^{-1}(\Gamma)$ and $\pi^{-1}(\Gamma')$ are irreducible curves of positive 
geometric 
genus. This is the general strategy that we have adopted to prove Part (2) in the above lemmas. 
We now indicate how to find such curves $\pi^{-1}(\Gamma)$ and $\pi^{-1}(\Gamma')$ of positive genus, and leave it to the reader to 
completely justify our computations.

\medskip
\noindent 
{\it Proof of Lemma {\rm \ref{genus.small} (2)}.}
When $g=1$, the reduction of $X/K(t)$ over $R:=K[t]_{tK[t]}$ is multiplicative already.
Suppose now that $g$ is even.  Let $F=K(s)$ with $s^2=t$. Then the curve $X_F/F$  has stable reduction over $R[s]$ 
with reduction consisting in the union of two smooth curves of genus $g/2$ meeting at one point. 
This can be seen using the two charts $y^2=x(x^g+1)(x^g+s^{2g})$, and $(y/x^{g+1})^2=(1/x) (s^{2g}(1/x)^g+1)((1/x)^g+1)$.
In both charts, we find in reduction a curve $y^2=x(x^g+1)$, of genus $g/2$.
This shows that the stable reduction of $X/K(t)$ cannot be good.

If $g>1$ is odd, then $X/K(t)$ has stable reduction over $R$ with reduction equal to the union of two smooth curves of genus $(g-1)/2$ meeting at two points. 
This can be seen using the two charts $y^2=x(x^g+1)(x^g+t^{g})$, and 
$(y/x^{g+1})^2=(1/x) (t^g(1/x)^g+1)((1/x)^g+1)$. 
 This shows that the stable reduction of $X/K(t)$ cannot be good.
\qed

\medskip
\noindent 
{\it Proof of Lemma {\rm \ref{lem.semistable}} when $N \geq 5$.}
The hyperelliptic curve over $K(u)$ given by the equation $y^2=x^N+ux^k+1$
is isomorphic over $K(s)$ with $s^2=1/u$ to the curve given by $z^2=s^2x^N+x^k+s^2$. This latter curve
in turn is isomorphic to the curve $X_F/F$ where $F=K(t)(s)$ with $s^2=t$.
Thus indeed, Part (2) implies Part (1) when $g \geq 2$. 

In the case where $N=4$ and $k=2$, the curve has genus $1$ and its semi-stable reduction is multiplicative. 
In general, we work over $F=K(s)$, where $s^{2k(N-k)}=t$ and consider the following two charts:
first, $y_1^2=s^{2N(N-k)}x_1^N+ x_1^k+1$ with $x_1=x/s^{k(N-k)}$ and $y_1=y/s^{k(N-k)}$, which in reduction gives a curve $y_1^2=x_1^k+1$ of genus $g_1:=\lfloor (k-1)/2\rfloor$
(we use here that $k$ is coprime to ${\rm char}(K)$). Second, $y_2^2= x_2^k (x_2^{N-k} +1) + s^{2kN}$ and $x_2\ne 0$, 
with $x_2=s^{2k}x$ and $y_2=s^{k^2}y$, 
which gives in reduction $y_2^2= x_2^k (x_2^{N-k} +1)$, of genus $g_2:=\lfloor (N-k-1)/2\rfloor$ if $k$ is even, and $g_2:=\lfloor (N-k)/2\rfloor$ is $k$ is odd. Note that we exclude the points with $x_2=0$ which 
correspond to the intersections points with the above irreducible component. 
Thus we found two components of positive genus except when $k=2$, or $k=N-2$ and $N$ is even. Note that in the latter case, we can reduce to the case $k=2$ by a change of variables. Let us then consider the case $k=2$, where $g_2=g(X)-1$. In this case, 
the stable reduction consists of the curve of genus $g_2$ with one node. 
\qed

\begin{emp} \label{orbit3} We now consider families of curves of genus $g=1$.
Let $V$ denote the open subset of $\Spec K[a,b,c,d,e]$, respectively of $\Spec K[b,c,d,e]$,
where the polynomial $\ell(x)=ax^4+bx^3+cx^2+dx+e$ is separable of degree $4$, respectively $3$.
We explicitly define below the $K$-morphism 
$$j: V\longrightarrow \A^1_K,$$ 
which maps a separable polynomial $\ell(x)$ of degree $n$ to the $j$-invariant
of the Jacobian of the  curve $X/K$ of genus $1$ defined by $y^2=\ell(x)$.

Given $\ell(x)= ax^4+bx^3+cx^2+dx+e$, we have the usual invariants
$I:= 12ae -3bd + c^2$ and $J:=72ace+ 9bcd-27ad^2 - 27eb^2-2c^3$ (see, e.g., \cite{Jac}, or \cite{CF}). 
The quantity $4I^2-J^3$ is equal to $27$ times the discriminant of $\ell(x)$.
 When ${\rm char}(K) \neq 2$, we let the $j$-invariant be $$j:=64 \frac{4I^3}{(4I^3-J^2)/27}.$$
 
\begin{claim}
 The curves $y^2=\ell_1(x)$ and $y^2=\ell_2(x)$ are not isomorphic over $\overline{K}$ if the $j$-invariants associated
 to the polynomials $\ell_1(x)$ and $\ell_2(x)$ are not equal. 
\end{claim}
Indeed, this is well known when ${\rm char}(K) \neq 2,3$:
 When $\deg(\ell(x))=3$, this $j$-invariant is the $j$-invariant of the elliptic curve
given by the Weierstrass equation $y^2= \ell(x)$. When  $\deg(\ell(x))=4$, the curve $X/K$ given by $y^2=\ell(x)$ has genus $1$, and its Jacobian $E/K$ has
$j$-invariant $j$. In this case, the elliptic curve $E/K$ is given by $y^2=x^3-27Ix-27J$ and a natural $K$-morphism $X \to E$ of degree $4$ is given explicitly in 
\cite{Jac} or \cite{CF}, Remarks, 3., page 681. Clearly, the curves $y^2=\ell_1(x)$ and $y^2=\ell_2(x)$ are not isomorphic over $K$ if their Jacobians are not isomorphic.

Assume now that ${\rm char}(K) =3$. Let $R=W(\overline{K})$ denote the Witt ring (a complete discrete local ring of characteristic $0$) associated with the algebraic closure of $K$. Let $(\pi)$ denote its maximal ideal.
Lift the curve $X/K$ given by the equation $y^2=\ell(x) $ to a curve $\tilde X/R$ given by an equation $y^2=\tilde\ell(x)$ with  $\tilde\ell(x) \in R[x]$. 
More precisely, we obtain a smooth proper morphism $\tilde X \to {\rm Spec } R$, and associated to this family of smooth proper curves of genus $1$ is 
its Jacobian fibration $\tilde J \to {\rm Spec } R$. The latter has a section and is a smooth proper family of elliptic curves.
Since $R$ has characteristic $0$, we can compute the $j$-invariant of the Jacobian of the generic fiber of $\tilde X$ (that is, of the generic fiber 
of $\tilde J \to {\rm Spec} R$) using the equation $y^2=\tilde\ell(x)$.
The reduction modulo $\pi$ of $j(\tilde X)$ is the $j$-invariant attached to $\ell(x)$, given by the formula $j:=64 \frac{4I^3}{(4I^3-J^2)/27}$, and this element of $K^*$
is the $j$-invariant of the Jacobian of $X$. 
\end{emp}

\medskip
\noindent 
{\it Proof of {\rm \ref{genus.small} (1)} when $g=1$, and {\rm \ref{lem.semistable} (1)} when $N =4$.} In both cases, since $V$ is irreducible, 
it suffices to show that the morphism $j:V \to {\mathbb A}^1_K$ is not constant. We leave this verification to the reader.
\end{section}
\begin{section}{In all characteristics} \label{char2}

Let $K$ be any field, and let $L/K$ be any finite separable extension of degree $d$. Let $g \geq 1$ be any integer.
We obtained in Theorem \ref{pro.boundgeneral} the existence of a hyperelliptic curve of genus $g$ with a new point over $L$ when $g \geq d/4$, 
under the additional condition that ${\rm char}(K) \neq 2$. We revisit this question in this section when $K$ is any infinite field 
of any characteristic, and obtain the existence of a hyperelliptic curve with a new point over $L$ mostly only when $g \geq d/2-3$.
The method of proof in the case where ${\rm char}(K) \neq 2$ consisted of exploiting the existence of the approximate square root of a polynomial,
while in this section, we only exploit the fact that a finite separable extension can be generated by an element with null trace.
We start with a straightforward lemma.

\begin{lemma} \label{thm.car2bound}
Let $K$ be any field. Let $L_1/K, \dots, L_t/K$ be  
separable extensions of degree $d_i$.
Set $d:=\sum_{i=1}^t d_i$. Let $g \geq \lfloor d/2\rfloor$ be an integer.
Then there exists a hyperelliptic curve $X/K$ of genus $g$ with a $K$-rational point and a new point in $X(L_i)$ for each $i=1,\dots, t$, 
\end{lemma}

\proof By removing some $L_i$'s if necessary we can suppose that the extensions are pairwise non-isomorphic. Choose an element $\alpha_i \in L_i$ with $L_i=K(\alpha_i)$. Let $m(x)=\prod_i m_i(x) \in K[x]$ denote the product of the minimal (monic) polynomial $m_i(x)$ of $\alpha_i$ over $K$.
 By hypothesis, $\deg m(x) = d$ and $m(x)$ is separable.  Consider the hyperelliptic curve $X/K$ given by the affine equation
$y^2=m(x)f(x)$ when ${\rm char}(K) \neq 2$, and $y^2 +y = m(x)f(x)$ when ${\rm char}(K) = 2$, 
where $f(x) \in K[x]$ is a separable polynomial coprime to $m(x)$ of the following degree  (we leave it to the reader to check that such a polynomial 
exists even when $K$ is finite):
when $d$ is even, we choose $\deg(f) \geq 1$ odd, so that the genus of $X/K$ is $g$ with $2g+1= d + \deg(f)$.
When $d$ is odd, we choose $\deg(f) \geq 0$ even, so that the genus of $X/K$ is $g$ with $2g+1= d + \deg(f)$.
The curve $X/K$ has the obvious new points $P_i=(\alpha_i,0)$ in $X(L_i)$, and one $K$-rational point at infinity. 
\qed

\begin{remark} When ${\rm char}(K)=2$ in \ref{thm.car2bound}, we can take  the $L_i/K$ to be only simple over $K$, and not necessarily separable.
\end{remark}

The above lemma is slightly improved in our next theorem. In view of Theorem \ref{pro.boundgeneral}, this is a useful improvement only 
in the case where ${\rm char}(K)=2$. 

\begin{theorem} \label{pro.boundgeneral2} Let $K$ be any infinite field. For each $i=1,\dots, t$, let $L_i/K$ be a  separable extension of degree $d_i$ (the extensions $L_i$ need not be distinct). Let $d:=\sum_{i=1}^t d_i $. Assume $d \geq 7$.  
Then there exist infinitely many 
hyperelliptic curves $X/K$ of genus $g=\lfloor (d-5)/2\rfloor$, pairwise non-isomorphic over $\overline{K}$, such that the following are true.
\begin{enumerate}[\rm (a)]
\item The curve $X$ contains distinct  new points $P_1\in X(L_1),
  \dots, P_t\in X(L_t)$. More precisely, let $Q_i$ denote the image of the point $P_i: \Spec L_i \to X$.
  Then $Q_1,\dots, Q_t$ are $t$ distinct closed   points of $X$.
\item The curve $X/K$ has at least two additional $K$-rational points distinct from $Q_1,\dots, Q_t$. 
When  $d $ is odd, then $X/K$ has at least three additional $K$-rational points distinct from $Q_1,\dots, Q_t$.
\end{enumerate}
\end{theorem} 
\proof
Let us explain the strategy of proof. 
Choose elements $\beta_i \in L_i$ with $L_i=K(\beta_i)$. Let $m(x):=\prod_i m_i(x) \in K[x]$ denote the product of the minimal (monic) polynomial $m_i(x)$ of $\beta_i$ over $K$. When $d=2n$, write $m(x) = x^d + a_{d-1} x^{d-1} + \dots + a_1 x + a_0$.  By construction, $a_0 \neq 0$. 
Our goal is to show the existence of such elements with $a_{d-1}=0$. 
Consider then 
the equation
 \begin{equation}  \label{eq.ell2} 
y^2+(a_{2n-2}x^{n-2}+ \dots + a_n)y=-a_{n-1}x^{n-1}-\dots -a_1 x-a_0.
\end{equation}
If its discriminant is non-zero, it defines a hyperelliptic curve of
genus $g=d/2-3$ with two distinct rational points at infinity ($n\ge
2$) and 
the new $L$-points $(\beta_i, \beta_i^n)$ for each $i=1,\dots,t$. 

 When $d=2n-1$ is odd, we proceed similarly and find elements $\beta_i$ such that we can write $xm(x) = x^{2n}+ a_{2n-2} x^{2n-2} + \dots + a_1 x$ with $a_1 \neq 0$. We then consider the equation 
  \begin{equation}  \label{eq.ell2odd} 
y^2+(a_{2n-2}x^{n-2}+ \dots + a_n)y=-a_{n-1}x^{n-1}-\dots - a_1 x.
\end{equation} 
If its discriminant is non-zero, it defines a hyperelliptic curve of
genus $g=(d+1)/2-3$ with two distinct rational points at infinity and 
a third rational point $(0,0)$. Moreover, $(\beta_i, \beta_i^n)$ for
each $i=1,\dots,t$, are new $L$-points on the curve. 

\begin{emp} \label{trace0}
 We start as in the  
proof of Theorem  \ref{pro.boundgeneral} in \ref{proof}.
For each $i=1,\dots, t$, fix a generator $\alpha_i \in L_i$, so that $L_i=K(\alpha_i)$. 
For each $i$, consider the morphism $f_{\alpha_i} :=\A^{d_i}_K \to \A^{d_i}_K$ defined in \eqref{mor}. 
Identify $\prod_{i=1}^t \A^{d_i}_K$ with $\A^{d}_K$, and consider 
$$ f:= f_{\alpha_1} \times \dots \times f_{\alpha_t}: \A^{d}_K \longrightarrow \A^{d}_K.$$
Using \ref{pro.finite}, we find that $f$ is surjective. 
Define an affine surjective morphism $$\mu: \A^{d}_K =\prod_{i=1}^{t} \A^{d_i}_K\longrightarrow \A^{d}_K$$
such that on $\overline{K}$-rational points, 
$$(a_0^{(1)},\dots, a^{(1)}_{d_1-1}, \dots,  a_0^{(t)},\ldots, a^{(t)}_{d_t-1})
\longmapsto (a_0, a_1,\dots,a_{d-1})$$ 
with the property that 
\begin{multline*} a_0+a_1x+ \dots+a_{d-1}x^{d-1}+x^d = \\ 
(a_0^{(1)}+a_1^{(1)}x + \dots + a^{(1)}_{d_1-1}x^{d_1-1}+x^{d_1}) \cdot \ldots  \cdot (a_0^{(t)}+ \dots +a^{(t)}_{d_t-1}x^{d_t-1} + x^{d_t}).
\end{multline*}

In the target space ${\mathbb A}_K^d$ with coordinates $(a_0, a_1,\dots,a_{d-1})$,
consider the hyperplane $H'$ given by $a_{d-1}=0$. Since each trace map ${\rm Tr}_{L_i/K}$ is $K$-linear,  the preimage of $H'$ 
under  $\mu \circ f$ is again a hyperplane defined over $K$ in the source space ${\mathbb A}_K^d$ (given in the appropriate coordinates by the equation 
$\sum_{i=1}^t {\rm Tr}_{L_i/K}=0$).

Consider the restriction $(\mu \circ f)_{\mid H}: H \to H'$. 
In $H'$, consider the open subspace $V$ 
where the discriminant of the equation  \eqref{eq.ell2} when 
$d$ is even and \eqref{eq.ell2odd} when $d$ is odd, is non-zero. 

It is easy to check that $V$ is non-empty 
(for instance in characteristic $2$, the equations $y^2+x^{n-2}y=x+1$ when $d$ is even and 
$y^2+x^{n-2}y=x$ when $d$ is odd have non-zero discriminants) and, hence, is
dense in $H'$. 
Consider the preimage of $V$ under the morphism $(\mu \circ f)_{\mid H}$. As in the proof
of   Theorem \ref{pro.boundgeneral}, we   argue when $K$ is infinite
that we can find infinitely many rational points in this preimage that correspond to $(\beta_1,
\dots, \beta_t) \in \prod_i L_i$
with $L_i=K(\beta_i)$ and such that the equations \eqref{eq.ell2}
or \eqref{eq.ell2odd} 
define smooth hyperelliptic curves. We leave the details of the proof to the reader. \qed
 \end{emp}

\smallskip
Our next theorem uses an idea already found in \cite{Roh}, and improves Theorem \ref{pro.boundgeneral2}
when $d=9$. Theorem \ref{thm.d=9} provides an alternate proof, in all characteristics, of Theorem \ref{pro.boundgeneral} when $d \leq 9$.
The proof of Part (b) is omitted below, as it is very similar to the proof of the analogous statement in  Theorem \ref{pro.boundgeneral}.

\begin{theorem} \label{thm.d=9} 
Keep the hypotheses of Theorem {\rm \ref{pro.boundgeneral2}} and suppose that
$d=9$. 
\begin{enumerate}[\rm (a)]
\item Then there exist infinitely many elliptic curves $X/K$, pairwise distinct over $\overline{K}$, such that
the statement {\rm (a)} in Theorem {\rm \ref{pro.boundgeneral2}} holds. 
 
\item Let $F$ be any finite extension of $K$ containing the $L_i$'s.  
Then the image of at least one of the $P_i$'s in  
$X(F)$ 
has either order larger than $N$ or is of infinite order.
\end{enumerate}
\end{theorem} 
\proof
Let us explain the strategy of proof. 
Choose elements $\beta_i \in L_i^*$ with $L_i=K(\beta_i)$. Let
$m(x):=\prod_i m_i(x) \in K[x]$ denote the product of the minimal
(monic) polynomial $m_i(x)$ of $\beta_i$ over $K$. Write $m(x) = x^9 +
a_{8} x^{8} + \dots + a_1 x + a_0$.  
Then consider the cubic plane curve  $X/K$ given by the equation:
$$ z^3 +  a_7 xz^2+ a_6 yz^2+ a_5 x^2z + a_4 xyz + a_3 y^2z +a_2 x^2y + a_1xy^2+ a_0y^3.$$
This curve has the $K$-rational point $(1:0:0)$, and this point is
smooth if $a_2\neq 0$ or $a_5 \neq 0$. 
By construction, this cubic curve has the new point $(\beta_i^{-2}: \beta_i^{-3} :1) \in X(L_i)$ for all $i=1,\dots,t$ when $a_8=0$.
One can show, as we did in  \ref{trace0}, that it is possible to choose the
elements $\beta_i$ such that $a_8=0$ and such that the above plane curve is
smooth, so that it defines an elliptic curve over $K$. We leave the details of the proof to the reader. 
\qed

\begin{remark}
Given a separable extension $L/K$ of degree $d > 9$, we do not know whether there always exists
an elliptic curve $X/K$ with a new point over $L$. When $d=10$, Proposition \ref{pro.boundgeneral2} produces only a curve $X/K$ of genus $g=2$ 
with a new point over $L$.
\end{remark}
\end{section}

\begin{section}{Finite fields and large fields}
We consider in this section fields $K$
for which we can, given any finite separable extension $L/K$ and integer $g\geq 1$, prove the existence of a smooth projective geometrically connected curve $X/K$ of genus $g$ with a new point over $L$.
The case $L=K$ holds for any $K$ and any $g \geq 1$, as  can be easily seen by exhibiting appropriate hyperelliptic curves.

\begin{proposition} \label{thm.finite} Let $K$ be a finite field. Let $L/K$ be any finite extension. Let $g \geq 1$.
 Then there exists a smooth projective geometrically connected curve $X/K$ of genus $g$ with a new point over $L$.
\end{proposition}

\proof Let $K={\mathbb F}_q$ and let $L={\mathbb F}_{q^d}$.
As pointed out above, the case $d=1$ is easy. 
Assume now that $d=\ell^s$ is a positive power of a single prime $\ell$. 
Recall  the Weil bounds for a smooth projective geometrically
connected curve $X/{\mathbb F}_q$ and any $d\ge 1$:
$$q^d+1-2g\sqrt{q^d} \leq |X({\mathbb F}_{q^d})| \leq
q^d+1+2g\sqrt{q^d}.$$
If the inequality $$q^{d/\ell}+1+2g\sqrt{q^{d/\ell}} < q^d+1-2g\sqrt{q^d}$$
is satisfied, then every smooth projective geometrically connected curve of genus $g$ has a new point over ${\mathbb F}_{q^d}$.
Otherwise, we have 
$$2g \geq \frac{q^d -q^{d/\ell}}{\sqrt{q^d}+\sqrt{q^{d/\ell}}}.$$
Clearly, we have $ {\sqrt{q^d}-\sqrt{q^{d/\ell}}} \geq d-1$ (except 
when $q=2$ and $d \leq 4$), so that $2g \geq d-1$. We can then  apply \ref{thm.car2bound} to all cases except when $q=2$, $d=4$, and $g=1$,
to find that there exists a curve of genus $g \geq 1$ over  ${\mathbb F}_{q}$ with a new point over $L={\mathbb F}_{q^d}$.
In the case where $q=2$, $d=4$ and $g=1$, 
we consider the elliptic curve 
$y^2+y=x^3+x+1$. Let $\alpha \in {\mathbb F}_{16}$ be a root of 
$x^4+x^3+x^2+x+1\in \mathbb F_2[x]$ (which is irreducible). 
Then $(\alpha, \alpha^2)$ is a new point in $\mathbb F_{16}$. 

Assume now that $d$ is divisible by exactly $m>1$ distinct primes 
$\ell_1, \dots, \ell_m$. If the inequality
\begin{equation}
  \label{eq:toto}
  \sum_{i=1}^m \left(q^{d/\ell_i}+1+2g\sqrt{q^{d/\ell_i}}\right) <
  q^d+1-2g\sqrt{q^d}
\end{equation}
holds, then every smooth projective geometrically connected curve of genus $g$ has a new point over ${\mathbb F}_{q^d}$.
If the inequality \eqref{eq:toto} is not satisfied, then
\begin{multline*}
2g 
\geq \frac{q^d - \sum_{i=1}^m q^{d/\ell_i} - (m-1)}{\sqrt{q^d}+ \sum_{i=1}^m \sqrt{q^{d/\ell_i}}} 
\geq \frac{q^d - \sum_{i=1}^m q^{d/\ell_i} - (m-1)}{(m+1)q^{d/2}} 
\geq q^{d/2}/(m+1) -1.\\ 
\end{multline*} 
Except when $q=2$ and $d=6$, we find that  $q^{d/2}/(m+1) -1 \geq d-1$, so that $2g \geq d-1$.
When $q=2$ and $d=6$, a direct computation finds that $2g \geq 5.77$, so that again $2g\geq d-1$. We can apply \ref{thm.car2bound} again to conclude. 
\qed

\medskip
Recall that a field $K$ is called  {\it large} if every geometrically integral scheme of finite type $X/K$ with a $K$-rational smooth point 
is such that $X(K)$ is Zariski-dense in $X$.
 PAC fields, and 
fields of fractions of domains which are Henselian with respect to a nontrivial ideal,
are large 
(see \cite{PopH}, 1.1, and \cite{PopSurvey}, 1.A). 

\begin{proposition}  \label{thm.PAC} Let $K$ be a large field. 
Let $X/K$ be a geometrically integral scheme of finite type of positive
dimension with a smooth $K$-rational point. Let $L/K$ be a finite separable extension. 
Then $X/K$ has infinitely many new points over $L$.
\end{proposition}

\proof Since we assume that $X/K$ is geometrically integral of finite type with a smooth $K$-point, we can find in $X$ a geometrically integral smooth open affine subscheme of finite type which contains a smooth $K$-point (\cite{BLR}, 2.3/16). Thus, it suffices
to prove the proposition when $X$ is affine and smooth of finite type. 
Since we assume that $L/K$ is  separable, 
there exist only finitely many proper subfields $F$ of $L$ containing $K$. 
For each such subfield $F$, consider the Weil restriction $W_F:={\rm Res}_{F/K}(X_F)$ of $X_F/F$ from $F$ to $K$, which exists since $X/K$ is affine (\cite{BLR}, 7.6, Theorem 4). 
Then we have a natural closed $K$-immersion $W_F \to W_L$.
By construction, this immersion induces on $K$-points the natural inclusion $W_F(K)=X(F) \subset X(L)=W_L(K)$. 

Since $X_L/L$ is geometrically integral and smooth over $L$, we find that $W_L/K$ is geometrically integral over $K$ 
(see Lemma \ref{lem.geomintegral}) 
and smooth over $K$ (\cite{BLR}, 7.6, Proposition 5). 
Since $X(K) $ is not empty by hypothesis, and since we have a closed immersion $X \to W_L$, we find that $W_L$ has a (smooth) $K$-rational point. 
It follows from the fact that $K$ is large that $W_L(K)$ is dense in $W_L$. 

Consider the non-empty open subscheme $V/K$ of $W_L/K$ obtained by removing the finitely many closed subschemes $W_F$ from $W_L$, where $F$ runs through all extensions of $K$ in $L$, $F \neq L$. Since $W_L(K)$ is dense in $W_L$, we find that $V(K)$ is dense in $V$. 
In particular, there exist infinitely many $K$-points on $V/K$, and this infinite set of $K$-points on $V$ corresponds to an infinite set of new $L$-points in $X(L)$.
\qed

\smallskip As we noted at the beginning
of this section, there always exists a smooth projective  geometrically
connected curve over $K$ of any given genus $g \geq 1$  with a $K$-rational point, to which we can apply Proposition \ref{thm.PAC} when $K$ is large.
Proposition \ref{thm.PAC}, which can be applied when $K={\mathbb Q}_p$, indicates
that, given a finite extension $L/{\mathbb Q}$ and integer $g \geq 1$, there is no `local obstruction' to the existence of a smooth curve $X/{\mathbb Q}$
of genus $g$  with a ${\mathbb Q}$-rational point and with a new point over $L$. 

The following statement is probably well known, but we include a proof here because we did not find it proved in the literature.

\begin{lemma} \label{lem.geomintegral} 
Let $L/K$ be a finite separable extension. Assume that $Y/L$ is a geometrically integral quasi-projective variety $L$. 
Then ${\rm Res}_{L/K}(Y)/K$ is geometrically integral over $K$.
\end{lemma}

\proof Choose $\overline{K}$ an algebraic closure of $K$ with $K \subseteq L \subseteq \overline{K}$. 
Then there is a natural $\overline{K}$-isomorphism
$${\rm Res}_{L/K}(Y) \times_K \overline{K} \longrightarrow {\rm Res}_{L\otimes_K \overline{K}/\overline{K}}(Y\times_L (L\otimes_K \overline{K})).$$

Consider now a scheme $S$ and two $S$-schemes $S_1/S$ and $S_2/S$, with the induced morphism $S_1 \sqcup S_2 \to S$. Let $Z_1/S_1$ and $Z_2/S_2$ be two 
schemes. 
Then when all three Weil restrictions below exist, we  find a natural isomorphism of $S$-schemes
$${\rm Res}_{(S_1 \sqcup S_2)/S }(Z_1 \sqcup Z_2)  \longrightarrow {\rm Res}_{S_1/S}(Z_1) \times_S {\rm Res}_{S_2/S}(Z_2).$$
We use this isomorphism and the fact that $L/K$ is separable to find a natural  $\overline{K}$-isomorphism
$${\rm Res}_{L\otimes_K \overline{K}/\overline{K}}(Y\times_L (L\otimes_K \overline{K})) \longrightarrow (Y \times_L \overline{K})^{[L:K]}.$$
Since $Y \times_L \overline{K}$ is integral by hypothesis, so is $(Y \times_L \overline{K})^{[L:K]}$.
\qed

\end{section}

\end{document}